\documentclass{amsart}
\usepackage{amssymb}
\usepackage[mathscr]{eucal}
\usepackage{xspace}
\newcommand{\R}{\mathbb{R}}

\newcommand{\DD}{\mathcal{D}}

\newcommand{\N}{\mathbb{N}}
\newcommand{\C}{\mathbb{C}}
\newcommand{\LL}{\mathcal{L}}

\renewcommand{\H}{\mathcal{H}}
\newcommand{\CS}{\mathcal{S}}
\newcommand{\CR}{\mathcal{R}}
\newcommand{\CO}{\mathcal{O}}

\newcommand{\E}{{\mathcal E}} 
\newcommand{\F}{{\mathcal F}}

\newcommand{\po}{\partial}

\newcommand{\wto}{\rightharpoonup} 
\newcommand{\ve}{\varepsilon}

\newcommand{\la}{\langle}
\newcommand{\ra}{\rangle}

\newcommand{\vphi}{\varphi}
\newcommand{\loc}{{\text{\rm loc}}}
\newcommand{\X}{\times}

\renewcommand{\d}{\delta}
\renewcommand{\l}{\lambda}

\renewcommand{\a}{\alpha}
\renewcommand{\b}{\beta}

\newcommand{\s}{\sigma}
\newcommand{\g}{\gamma} 
\newcommand{\z}{\zeta}
\renewcommand{\k}{\kappa}
\newcommand{\sgn}{\text{\rm sgn}}

\newcommand{\Om}{\Omega}
\newcommand{\om}{\omega}

\newcommand{\supp}{\text{\rm supp}\,}

\newcommand{\M}{{\mathcal M}}

\renewcommand{\E}{{\mathcal E}}
\newcommand{\DM}{\mathcal D\mathcal M}

\renewcommand{\div}{\text{\rm div}\,}

\renewcommand{\supp}{\text{\rm supp}\,}

\newcommand{\abf}{{\mathbf a}}

\newcommand{\ul}{\underline}

\newcommand{\fF}{{\frak F}}

\newcommand{\cL}{{\mathcal L}}

\newcommand{\bff}{{\mathbf f}}
\newcommand{\bfa}{{\mathbf a}}

\newcommand{\esslim}{\operatorname{ess}\!\lim}
\newcommand{\BV}{\operatorname{BV}}

\newcommand{\Bbf}{\mathbf{B}}

\newcommand{\Lip}{\text{\rm Lip}}

\renewcommand{\S}{{\mathcal S}}
\newcommand{\B}{{\mathcal B}}

\newcommand{\K}{\mathcal{K}}

\newcommand{\dist}{\text{dist\,}}

\renewcommand{\Lip}{\text{Lip\,}}

\newcommand{\un}[1]{\underline{#1}}

\theoremstyle{plain}
\newtheorem{theorem}{Theorem}[section]

\newtheorem{lemma}{Lemma}[section]

\theoremstyle{definition}
\newtheorem{definition}{Definition}[section]
\theoremstyle{remark}

\newtheorem{remark}{Remark}[section]
\numberwithin{equation}{section}
\begin{document}

%Topmatter
\title[Boundary Value Problem for Degenerate Parabolic Equations]
{A Boundary Value Problem for a Class of Anisotropic Degenerate Parabolic-Hyperbolic Equations }
\author{Hermano Frid} 

\address{Instituto de Matem\'atica Pura e Aplicada - IMPA\\
         Estrada Dona Castorina, 110 \\
         Rio de Janeiro, RJ 22460-320, Brazil}
\email{hermano@impa.br}

\author{Yachun Li}

\address{School of Mathematical Sciences, MOE-LSC, and SHL-MAC,  Shanghai Jiao Tong University\\ Shanghai 200240, P.R.~China}

\email{ycli@sjtu.edu.cn}

\keywords{divergence-measure fields, normal traces, Gauss-Green theorem,
          product rule} 
\subjclass{Primary: 26B20,28C05, 35L65, 35B35; Secondary: 26B35, 26B12
            35L67}
\date{}
\thanks{}

\begin{abstract} 

We consider a mixed type boundary value problem for a class of degenerate parabolic-hyperbolic equations. Namely, we consider a Cartesian product domain and split its boundary into two parts. In one of them we impose a Dirichlet boundary condition; in the other, we impose a Neumann condition.  We apply a normal trace formula for $L^2$-divergence-measure fields to  prove a new strong trace property in the part of the boundary where the Neumann condition is imposed.  We prove existence and uniqueness of the entropy solution.   This is a revised corrected version of the paper published in  Arch.\ Ration.\ Mech.\ Anal.\ {\bf 226} (2017), no. 3, 975--1008.

\end{abstract}

\maketitle

\section{Introduction}\label{S:1}

We are concerned with the following mixed type boundary value problem for a class of anisotropic degenerate parabolic-hyperbolic equation. Let $\Om\subset\R^d$ be a bounded domain  of the form $\Om=\Om'\X\Om''$, with $\Om'\subset\R^{d'}$, $\Om''\subset\R^{d''}$ , $d'+d''=d$, bounded open sets with smooth boundaries. For $x\in\R^d$, let us denote $x=( x', x'')$, with $x'=(x_1,\cdots,x_{d'}) \in\R^{d'}$ and $x''=(x_{d'+1},\cdots,x_d)\in\R^{d''}$. We consider the following initial-boundary value problem for a degenerate parabolic-hyperbolic equation 
\begin{align}
&u_t+\nabla\cdot \bff(u)= \nabla_{x''}\cdot (B'(u) \nabla_{x''} u), \quad (t,x)\in (0,\infty)\X\Om,\label{e1.1} \\
&u(0,x)=u_0(x), \quad x\in\Om,\label{e1.2}\\
&u(t,x)=a_0(x), \quad x\in \Gamma'':=\Om'\X\po\Om'', \ t>0, \label{e1.3}\\
& \bff(u(t,x))\cdot \nu(x)= 0,\quad x\in \Gamma':=\po\Om'\X\Om'',\ t>0, \label{e1.4}
\end{align}
where  $\bff(u)=(f_1(u),\cdots,f_d(u))$, $f_i:\R\to\R$, $i=1,\cdots,d$,  smooth functions, we denote $\abf(u)=(a_1(u),\cdots,a_d(u)):=\bff'(u)$,  $B'(u)=(b_{ij}'(u))_{i,j=1}^d$, $b_{ij}'(u)=\frac{d}{du} b_{ij}(u)$, $b_{ij}=b_{ji}:\R\to\R$ , $i,j=1,\cdots,d$, smooth functions,  with $b_{ij}'(u)=0$, if $\min\{i,j\}\le d'$,  and
\begin{equation}\label{e1.B'}
b'(u)^2|\xi|^2 \le \sum_{i,j=d'+1}^db_{ij}'(u)\xi_i\xi_j\le \Lambda b'(u)^2|\xi|^2, \qquad \text{for all $\xi\in\R^{d''}$},
\end{equation}
for some smooth $b:\R\to\R$, with $b'(u)\ge0$ and some $\Lambda>0$. We also denote
 $$
 \nabla_{x''}:=(\underset{\text{$d'$ times}}{\underbrace{0,\cdots,0}},\partial_{x_{d'+1}},\cdots,\partial_{x_{d}}).
 $$  
and $\nu(x)$ is  the outward unit normal vector to $\po\Om$.  

Observe that $\nabla_{x''}B'(u)\nabla_{x''}u=\nabla_{x''}^2:B(u)$, with $B(u)=(b_{ij}(u))_{i=1}^d$, $b_{ij}(u)=0$, if $\min\{i,j\}\le d'$. Sometimes, whenever the context is such that what is meant is clear, we will also consider $B'(u)$ (or $B(u)$) as a $d''\X d''$ matrix and $\nabla_{x''}$ as $(\po_{x_{d'+1}},\cdots,\po_{x_d})$.  

{}From \eqref{e1.B'} we easily see that $b'(u)=0$ implies $b_{ij}'(u)=0$ for all $i,j=d'+1,\cdots,d$. Therefore, each $b_{ij}$ may be written as $b_{ij}(u)=\tilde b_{ij}(b(u))$ with $\tilde b_{ij}$ continuous. Moreover, \eqref{e1.B'} implies that $b'(u)^2\le b_{ii}'(u)\le \Lambda b'(u)^2$ and  $-(\Lambda-1)b'(u)^2\le b_{ij}'(u)\le (\Lambda-1)b'(u)^2$, if $i\ne j$, for all $u\in\R$ and $i,j=d'+1,\cdots,d$. Then, it follows  that the  $\tilde b_{ij}$ are in fact locally Lipschitz functions and so all the $b_{ij}$ may be written as locally Lipschitz functions of $b(u)$.  

Problems of the type of \eqref{e1.1}-\eqref{e1.4} may appear, for instance, in models in two-phase flow in porous medium where capillarity effects are considered only in certain directions.

We call
 $$
 \LL(\tau, \k, \xi)= i(\tau +\bfa(\xi)\cdot \k) + {\k''}^\top B'(\xi) \k''
 $$ 
 the symbol of \eqref{e1.1}, and
 $$
 \LL_0(\tau, \k, \xi)= i(\tau +\pi_{d'}(\bfa(\xi))\cdot \k') + {\k''}^\top B'(\xi) \k''
 $$ 
 the reduced symbol of \eqref{e1.1}, where  $\pi_{d'}(\bfa)=(a_1,\cdots,a_{d'})$ and we consider $B'(v)$ as a $d''\X d''$ matrix.  We assume that $\LL(\tau,\k,v)$ satisfies the conditions of Tadmor and Tao's averaging lemma~2.3 in \cite{TT}.
 Namely, for $J,\d>0$ and $\eta\in C_b^\infty(\R)$ nonnegative, let 
\begin{equation*}
\begin{aligned}
\Om_{\LL}^\eta(\tau,\eta;\d)&:=\{\xi\in \supp\eta\,:\, |\LL(i\tau, \k, \xi)|\le \d\},\\
\om_{\LL}^\eta(J;\d) &:= \sup_{\tiny\begin{matrix} (\tau,\k)\in\R^{d+1}\setminus\{0\}\\ |(\tau,\k)|\sim J\end{matrix}}|\Om_{\LL}^\eta(\tau, \k;\d)|
\end{aligned}
\end{equation*}
and $\LL_{\xi} :=\po_\xi\LL$ and we use $|\{\cdots\}|$ to denote the one-dimensional Lebesgue measure of $\{\cdots\}$. We suppose that there exist $\a\in (0,1)$, $\b>0$ and a measurable function  $\vartheta\in L_\loc^\infty(\R;[1,\infty))$ such that
\begin{equation}\label{e1.TT}
\begin{aligned}
\om_{\cL}^\eta(J;\d) &\lesssim_\eta \left(\frac{\d}{J^\b}\right)^\a,\\
\sup_{\tiny{\begin{matrix}(\tau,\k)\in\R^{d+1}\setminus\{0\}\\|(\tau,\k)|\sim J \end{matrix}}}\sup_{\xi\in\supp \eta}\frac{|\cL_{\xi}(i\tau,i\k;\xi)|}{\vartheta(\xi)}&\lesssim_\eta J^\b,\qquad \forall \d>0,\, J\gtrsim 1,
\end{aligned}
\end{equation}
where we employ the usual notation $x\lesssim y$, if $x\le Cy$, for some absolute constant $C>0$, and $x\sim y$, if $x\lesssim y$ and $y\lesssim x$, and $\lesssim_\eta$ means $\lesssim$ on the support of $\eta$.

The following example in the case where $d=2$, $d'=d''=1$,  is shown in corollary~4.5 in \cite{TT} to satisfy relations \eqref{e1.TT} which corresponds to (2.19) and (2.20) in \cite{TT} : 
$$
\rho_t(t,x)+\po_{x_1}(\frac1{\ell+1}\rho^{\ell+1}(t,x))= \po_{x_2}^2(\frac1{n+1}|\rho^n(t,x)|\rho(t,x)),
$$
where $\ell,n\in\N$ satisfy $n\ge 2\ell$. The same argument as in corollary~4.5 of \cite{TT} applies to the corresponding equation in any space dimension $d$,  with $d'=1$, $d''=d-1$, replacing $\po_{x_2}^2$ in the above equation by $\Delta_{x''}=\po_{x_2}^2+\cdots+\po_{x_d}^2$.

We observe that conditions \eqref{e1.TT}  imply  the weaker  non-degeneracy condition: For $(\tau,\k)\in\R^{d+1}$, with $\tau^2+|\k|^2=1$,  
\begin{equation}\label{e1.nondeg}
\left| \left\{\xi \in\supp\eta \,:\, |\tau+\pi_{d'}(\bfa)(\xi)\cdot\k'|^2+ \left( {\k''}^{\top} B'(\xi)\k''\right)^2 =0 \right \} \right|=0.
\end{equation}

We are going to seek  solutions of the initial-boundary value problem \eqref{e1.1}-\eqref{e1.4} which assume values in an interval, say $[u_{\min},u_{\max}]$, such that $u_0(x),a_0(x)\in [u_{\min},u_{\max}]$, and
\begin{equation}\label{e1.f}
\pi_{d'}(\bff)(u_{\min})=\pi_{d'}(\bff)(u_{\max})=0.
\end{equation}
In particular, we will assume that condition \eqref{e1.TT} holds for some $\eta \in C_c^\infty(\R)$ such that $\eta\equiv 1$ on $[u_{\min},u_{\max}]$. 

Several authors have contributed works dedicated to the study of degenerate parabolic equations, starting with Vol'pert and Hudjaev in \cite{VH}, giving the existence of solution for the Cauchy problem, whose extension to the Dirichlet boundary value problem appeared in \cite{WZ}. Uniqueness for the homogeneous Dirichlet problem, for the isotropic case, was only achieved  many years later by Carrillo in \cite{Ca}, using an extension of Kruzhkov's doubling of variables method \cite{Kr}. The result in \cite{Ca} was  extended to non-homogeneous Dirichlet data by Mascia, Porretta and Terracina in \cite{MPT}, by extending to the parabolic case ideas in \cite{BLN} further developed in \cite{O}.  In \cite{MPT}, in order to prove the uniqueness of solutions, use is made of the existence of the normal trace for $\DM^2$-vector fields which has been established in \cite{MPT} and \cite{CF2}, independently.   An $L^1$ theory for the Cauchy problem for anisotropic degenerate parabolic equations was established by Chen and Perthame \cite{CP}, based on the kinetic formulation (see \cite{PB}), and later also obtained using Kruzhkov's approach in \cite{CK}.  

The problem we address here combine two different types of boundary conditions. In the ``hyperbolic boundary'', $(0,T)\X \po\Gamma'$,  we impose the Neumann condition, and so, concerning the uniqueness of the solution, the treatment of this part of the boundary requires the use of the strong trace property first proved, in  the hyperbolic case, by Vasseur~\cite{Va}. This is  what was  done to prove the uniqueness of the solution of the Neumann problem  in the hyperbolic case in \cite{BFK}. Here, the major point concerning this part of the boundary is that the corresponding $\DM$-vector field is no longer $L^\infty$, but rather $L^2$, hence the strong trace theorem in \cite{Va} does not apply. However we prove an analogous strong trace theorem in this paper, which can be viewed as an extension of the one in \cite{Va}. We also impose a non-homogeneous Dirichlet  boundary condition in the ``parabolic boundary'', $(0,T)\X\Gamma''$. In order to treat this part of the boundary, we follow an approach  inspired by that in \cite{MPT}, and, while doing that,  we highlight the main points in the strategy and make  simplifications in the presentation and proofs in \cite{MPT}.  

This paper is only concerned with the case $d'>0$, since the case $d'=0$ is covered in \cite{MPT} and \cite{MV}, at least in the isotropic case. Although here we address the more general anisotropic case, as it will become clear,  the constraint \eqref{e1.B'} allows for proofs that are very similar to the isotropic case. In particular, condition \eqref{e1.TT} is not needed in the proof of the uniqueness result  if $d'=0$. In the latter case,  this condition can also be avoided in the proof of the existence of entropy solutions 
by adopting an approach based on the uniqueness of  measure-valued type solutions  as in \cite{MV}.   
 
 We now prepare the way to state our main results. We introduce the functions
 \begin{equation}\label{e1.5}
 \begin{aligned}
 & F(u,v):=\sgn (u-v) (\bff(u)-\bff(v)),\\
 &\Bbf(u,v)=(\sgn(u-v)(b_{ij}(u)-b_{ij}(v)))_{i,j=1}^d\\
 & K_{x''}(u,v):=  \nabla_{x''}\cdot \Bbf(u,v) -F(u,v),\\
 & H_{x''}(u,v,w):= K_{x''}(u,v)+K_{x''}(u,w)-K_{x''}(w,v),
 \end{aligned}
 \end{equation}
where  $\nabla_{x''}\cdot \Bbf(u,v)$ is the $d$-vector with components 
\begin{multline*}
(\nabla_{x''}\cdot \Bbf (u,v))_j\\=\begin{cases} 0, &\text{if $j\le d'$}\\ \sum\limits_{i=d'+1}^d\po_{x_i}(\sgn(u-v)(b_{ij}(u)-b_{ij}(v))), &\text{for $d'+1\le j\le d$}.\end{cases}
\end{multline*}
We also define
\begin{equation}\label{e1.6}
A(u,v,w)=|u-v|+|u-w|-|w-v|.
\end{equation}

 To address the Dirichlet condition, in order to take advantage of the fact that $\po\Om''$  is locally the graph of a $C^2$ function, we introduce a system of balls $\B''$, with the following property. For each $B''\in\B''$, $B''=B''(x_0'',r)$, a  ball in $\R^{d''}$ of radius $r>0$ around an arbitrary point $x_0''\in\po\Om''$, we have that for some $\gamma\in\Lip(\R^{d''-1})$,
\begin{equation}\label{e1.8B}
B''\cap \Om''=\{(\bar y'',y_{d})\in B''\,:\, y_{d}<\gamma(\bar y''),\ \bar y''=(y_{d'+1},\cdots,y_{d-1})\in \R^{d''-1}\},
\end{equation}
where the coordinate system $(y_{d'+1},\cdots,y_{d})$ is obtained from the original $(x_{d'+1},\\ \cdots,x_{d})$ by  relabeling, reorienting and translation.  By relabeling we mean a permutation of the coordinates and by reorienting we mean changing the orientation of one of the coordinate axes.

We assume that  $a_0\in C_0^2(\Om'\X\po\Om'')$. For each $B''\in\B''$, we take the following extension of $a_0$ to $\Om'\X B''$:   
\begin{multline}\label{e1.a00}
a_0(x',\bar y'',y_d)=a_0(x',\bar y'',\gamma(\bar y'')), \\ \text{for $(x',y'')\in \Om'\X B''$, with  $\gamma\in \Lip(\R^{d''-1})$ as in \eqref{e1.8B}}. 
\end{multline}

We also assume that $a_0$, extended to $\Om'\X B''$ as in \eqref{e1.a00}, satisfies
\begin{multline}\label{e1.C}
b_{ij}(u)\nabla_{x''}b(a_0(x)) =0,\\ \text{for $i\ne j$, $i,j=d'+1,\cdots,d$, for all $u\in [a,b]$ and $x\in \Om'\X B''$, for any $B''\in\B''$},
\end{multline}
by which it follows that either $b_{ij}(u)=0$, for $i\ne j$, $u\in[a,b]$,  or $\nabla_{x''}b(a_0(x))=0$, $x\in \Om'\X B''$, for any $B''\in\B''$. 

Further, for all $B''\in\B''$  and $k\in\R$, we assume that 
\begin{equation} \label{e1.a0}
\sgn (a_0-k)\in \BV(\Om'\X B''),\quad \text{for all $i,j=1,\cdots,d$, $k\in\R$}.
\end{equation}
In particular, we have that $K_{x''}(a_0,k)\in \BV(\Om'\X B'')$.

We observe that since $B'(u)$ is a non-negative symmetric matrix it possesses a square root non-negative symmetric matrix, $\s(u):= (\b_{ik}'(u))_{i,k=d'+1}^d$, that is,
\begin{equation}\label{e1.bij}
b_{ij}'(u)=\sum_{k=d'+1}^d \b_{ik}'(u)\b_{jk}'(u),
\end{equation}
with $\b_{ik}'(u)=\frac{d}{du}\b_{ik}(u)$, with $\b_{ik}=\b_{ki}:\R\to\R$ smooth functions.  We also observe that \eqref{e1.B'} implies that for smooth $v(x)$ we have
\begin{equation}\label{e1.b1}
|\nabla_{x''} b(v)|^2\le \sum_{k=d'+1}^d\left(\sum_{i=d'+1}^d \po_{x_i}\b_{ik}(v)\right)^2\le\Lambda |\nabla_{x''} b(v)|^2.
\end{equation}
More generally, given any smooth function $\eta(u)$, \eqref{e1.B'} implies
\begin{equation}\label{e1.b2}
|\nabla_{x''} b_\eta(v)|^2\le \sum_{k=d'+1}^d\left(\sum_{i=d'+1}^d \po_{x_i}\b_{ik\eta}(v)\right)^2\le\Lambda |\nabla_{x''} b_\eta(v)|^2.
\end{equation}
where $b_\eta'(u)=\eta'(u)b'(u)$ and $\b_{ij\eta}'(u)=\eta'(u)\b_{ij}'(u)$.

\begin{definition} \label{D:1.1} Assume that $u_0\in L^\infty(\Om)$.  Given $T>0$, for $U_T=(0,T)\X\Om$,  we say that a function $u\in L^\infty(U_T)$ is an entropy solution of the problem \eqref{e1.1}-\eqref{e1.4} if:
\begin{enumerate}
\item[{\bf(i)}] (Regularity) We have
\begin{equation}\label{e1.7}
\nabla_{x''} b(u)\in L^2(U_T);
\end{equation}
\item[{\bf(ii)}] (Entropy condition)  For all $\eta\in C^2(\R)$, with $\bff_\eta, B_\eta$ such that $\bff_\eta'=\eta'\bff'$ and $B_\eta'=\eta' B'$, and for all
$0\le \varphi\in C_0^\infty(U_T)$ 
\begin{multline}\label{e1.8E-1}
\int_{U_T}\{\eta(u)\po_t\varphi +\bff_\eta(u)\cdot\nabla \varphi-B_\eta(u):\nabla^2\varphi\}\,dx\,dt\\
 \ge \int_{U_T}\eta''(u)\sum_{k=d'+1}^d\left(\sum_{i=d'+1}^d\po_{x_i}\b_{ik}(u)\right)^2\varphi\,dx\,dt.
 \end{multline}
In particular, for all  $k\in\R$,
\begin{equation}\label{e1.8E}
\int_{U_T}\{|u-k|\po_t\varphi-K_{x''}(u,k)\cdot\nabla \varphi\}\,dx\,dt \ge 0.
\end{equation}

\item[{\bf(iii)}] (Neumann condition on $\Gamma'$) For all $\tilde \phi\in C_0^\infty((0,T)\X\R^{d'}\X\Om'')$,
\begin{equation}\label{e1.8N}
\int_{U_T}\{ u\po_t\tilde\phi +\bff(u)\cdot\nabla \tilde\phi-\nabla_{x''}\cdot B(u)\cdot\nabla_{x''}\tilde\phi\}\,dx\,dt=0,
\end{equation}
where $\nabla_{x''}\cdot B(u)$ is the $d$-vector whose $j$-th component is 0, if $j\le d'$, and $\sum_{i=d'+1}^d\po_{x_i}b_{ij}(u)$, for $d'+1\le j\le d$. Observe that the test function $\tilde\phi$ may not vanish over $\Gamma'_T:=(0,T)\X\Gamma'$. 

\item[{\bf(iv)}] (Dirichlet condition on $\Gamma''$) For each $B''\in\B''$,  for $\mu_0:=|\div K_{x''}(a_0,k)|$, i.e., the total variation measure of the (signed) measure over $\Om'\X B''$, 
$\div  K_{x''}(a_0,k)$,  and some constant $C_*>0$ depending only on $\bff, b, a_0$,  we have, for all $0\le\tilde \varphi\in C_0^\infty((0,T)\X\Om'\X B'')$, 
\begin{multline}\label{e1.9}
\int_{U_T}\{|u(t,x)-a_0(x)|\po_t\tilde\varphi-K_{x''}(u(t,x),a_0(x))\cdot\nabla\tilde \varphi\}\,dx\, dt\\ \ge -C_*\int_{U_T}  \tilde \varphi\,dx\,dt,
\end{multline}
and, for all $k\in\R$,
\begin{multline}\label{e1.9'}
\int_{U_T}\{A(u(t,x),k,a_0(x))\po_t\tilde\varphi-H_{x''}(u(t,x),k,a_0(x))\cdot\nabla\tilde \varphi\}\,dx\, dt\\ \ge -C_*\int_{U_T}\tilde \varphi\,dx\,dt -\int_{U_T}\tilde \varphi\,d\mu_0(x)\,dt.
\end{multline}
Moreover,  
\begin{multline}\label{e1.8D'}
(t,x'')\in (0,T)\X\Om'' \mapsto \int_{\Om'}|b(u(t,x',x''))-b(a_0(x',x''))|\tilde\varphi(t,x',x'') \,dx' \\ \in L^2((0,T); H_0^1(\Om'')).
\end{multline}
\item[{\bf(v)}] (Initial condition)
\begin{equation}\label{e1.10}
\esslim_{t\to 0+} \int_{\Om} |u(t,x)-u_0(x)|\,dx=0.
\end{equation}
\end{enumerate}
\end{definition}

We can now state our main result concerning problem \eqref{e1.1}--\eqref{e1.4}.

\begin{theorem}\label{T:1.1} There exists a unique entropy solution of \eqref{e1.1}--\eqref{e1.4}.
\end{theorem}

Another main result of this paper is the following extension of the Strong Trace Theorem of Vasseur~\cite{Va}. The definition of strongly regular deformation is given in Definition~\ref{D:2.1}. This result is an essential tool for the proof of Theorem~\ref{T:1.1}.
  
 \begin{theorem}\label{T:1.2}  Let  $\Gamma_T'$, and let $K_{x''}$ be as above. Assume \eqref{e1.TT} holds, and let  $u(t,x)\in L^\infty(U_T)$ satisfy $\nabla_{x''}b(u)\in L^2(U_T)$ and ,  for all $0\le \varphi\in C_0^\infty(U_T)$ and $k\in\R$,
 \begin{equation}\label{e1.T12}
\int_{U_T}\{|u-k|\po_t\varphi-K_{x''}(u,k)\cdot\nabla \varphi\}\,dx\,dt\ge0.
\end{equation}
Then, there exists $u^\tau\in L^\infty(\Gamma_T')$  such that, for any  deformation of $\po\Om'$,  $\Psi': [0,1]\X\po \Om'\to \bar\Om'$, strongly regular over $\po\Om'$,  if  $\Psi: [0,1]\X\Gamma'_T \to \bar U_T$, is defined by $\Psi(s,t,x',x'')=(t,\Psi'(s,x'),x'')$, we have $u(\Psi(s,t,x))\to u^\tau(t,x)$, as $s\to0$, in $L^1(\Gamma_T')$.
\end{theorem}

 We will use here a slight extension of the concept of boundary layer sequence as  defined in \cite{MPT}, which is in agreement with the type of sequence used for the same purpose in the theory of divergence-measure fields, as recalled in  Section~\ref{S:3}. Namely, we say that $\z'_\d$ is a {\em $\Om'$-boundary layer sequence} if, for each $\d>0$, $\z'_\d\in \Lip(\bar\Om')$,  $0\le\z'_\d\le1$, $\z'_\d(x')\to1$ for every $x'\in\Om'$, as $\d\to0$, and $\z'_\d=0$, on $\po\Om'$.  We define an
 $\Om''$-boundary layer sequence $\z''_\d$ in a totally similar manner replacing $\Om'$ by $\Om''$ and $x'$ by $x''$.

 \begin{remark} \label{R:1.1} Let $\Psi: [0,1]\X \po\Om''\to\bar\Om''$ be any Lipschitz deformation for $\po\Om''$ (see, Definition~\ref{D:2.1}), and let $h:\bar\Om''\to [0,1]$ be the associated level set function, that is, $h(x'')=s$, for $x''\in \Psi(s,\po\Om'')$, and $h(x'')=1$, for $x''\in\Om''\setminus\Psi([0,1]\X\po\Om'')$. Then, 
\begin{equation}\label{e1.R1}
 \z_\d(x'')=\frac1{\d}\min\{\d,h(x'')\},
 \end{equation}
 for $0<\d<1$, is a $\Om''$-boundary layer sequence, which we will call the $\Om''$-{\em level set boundary layer sequence} associated with the deformation $\Psi$. Assuming that each $\Psi(s, \po\Om'')$, $s\in[0,1]$, is at least of class $C^{1,1}$, we have the following
 \begin{equation}\label{e1.R1'}
\begin{aligned}
& \nabla\z_{\d}(x'')=-\frac1{\d}\chi_{{}_{\{0< \z_{\d}(x'')<1\}}}(x'') N(x''),\\
&\nabla_{x''}^2 \z_{\d}(x'')|\Om''= -\frac{N(x'')\otimes \nu(x'')}{\d}d\H^{d''-1}(x'')\lfloor\Psi(\d,\po\Om'') \\
&\qquad\qquad\qquad+ -\frac1{\d}\chi_{{}_{\{0< \z_{\d}(x'')<1\}}}(x'') \nabla_{x''} N(x'') ,
\end{aligned}
\end{equation}
where $N(x'')=\l(x'')\nu(x'')$, $\nu(x'')$ denotes the outward unity normal to $\Psi(\d\z_{\d}(x''),\po\Om'')$,  $\l(x'')$ is a positive Lipschitz function, and $\H^{d''-1}\lfloor \Psi(\d,\po\Om'')$ denotes the $(d''-1)$-dimensional Hausdorff measure restricted to the hyper-surface $\Psi(\d,\po\Om'')$. 
\end{remark}

 Concerning Definition~\ref{D:1.1}, we have the following. 
 
 \begin{lemma}\label{L:1.1} Condition \eqref{e1.9'} in Definition~\ref{D:1.1} implies that, for any $\Om''$-boundary layer sequence $\z''_\d$ and all $0\le \tilde\varphi\in C_0^\infty((0,T)\X \Om'\X B'')$, $B''$ as in \eqref{e1.8B},  we have
 \begin{equation}\label{e1.11}
 \liminf_{\d\to0}\int_{U_T}H_{x''}(u,k,a_0)\cdot\nabla\z''_\d(x'')\tilde\varphi(t,x)\,dx\,dt\ge0.
 \end{equation}
 Moreover,  Definition~\ref{D:1.1} implies that, for any $0\le \psi''\in C_0^\infty(B'')$, the distributions over $U_T$,  
 \begin{gather*}
\ell_1:= -\psi''(\po_t|u-a_0|- \nabla\cdot K_{x''}(u,a_0))\\
\ell_2:=- \psi''(\po_t A(u,k,a_0)-\nabla\cdot H_{x''}(u,k,a_0)),
 \end{gather*}
 are (signed) measures with finite total variation over $U_T$, uniformly for $k$ in bounded intervals. Furthermore, the distribution over $U_T$,
$$
\ell_0:= -\po_t|u-k|+\nabla\cdot K_{x''}(u,k),
$$
is a positive measure with finite total variation over $U_T$,  uniformly for $k$ in bounded intervals.    
 
  \end{lemma}
 
\begin{proof} For the proof of \eqref{e1.11} from \eqref{e1.9'},  we just substitute $\tilde \varphi$ in  \eqref{e1.9'} by $\tilde \varphi (1-\z_\d)$ where $\z_\d$ is any $\Om''$-boundary layer sequence, make $\d\to0$, observe that $(1-\z_\d)$ tends everywhere to zero in $U_T$, and we easily get \eqref{e1.11}. 

To prove the that $\ell_1$ and $\ell_2$ are signed measures with finite total variation over $\Om'\X(\Om''\cap B'')$, we substitute $\tilde\varphi$ in \eqref{e1.9} and \eqref{e1.9'} by
$\psi'' \tilde\varphi$ and after trivial manipulations we deduce
\begin{gather*}
\la \ell_1, \tilde\varphi\ra\ge -C(k,T,\Om,\bff,a_0, \|u\|_\infty,\|\nabla_{x''}b(u)\|_{L^2},  \|\psi'',\nabla\psi''\|_\infty)\|\tilde \varphi\|_\infty,\\
\la \ell_2, \tilde\varphi \ra\ge -C(k,T,\Om,\bff, a_0, \|u\|_\infty,\|\nabla_{x''}b(u)\|_{L^2},  \|\psi'',\nabla\psi''\|_\infty)\|\tilde \varphi\|_\infty \\
-\|\psi''\|_{\infty}\int_{U_T}\tilde\varphi\,d\mu_0(x)\,dt,
\end{gather*}
where $C(k,T,\Om,\bff,a_0, \|u\|_\infty,\|\nabla_{x''}b(u)\|_{L^2},  \|\psi'',\nabla\psi''\|_\infty)$ is a positive constant depending only on its arguments. 

Now, we are going to get the corresponding relations for $\ell_1$ and $\ell_2$ applied to $0\le\bar\varphi\in C^\infty(\bar U_T)$. We consider the $\Om'$-level set boundary layer sequence, $\z_\d'(x')$, associated with the level set function of the deformation $\Psi':[0,1]\X\po\Om'\to\bar\Om'$, given by $\Psi'(s,x')=x'-s\ve_0\nu(x')$, where $\nu(x')$ is the outward unity normal to $\po\Om'$ at $x'$, and $\ve_0$ is sufficiently small so that $\Psi'$ so defined is injective. We also consider a sequence $\xi_h(t)\in C_0^\infty((0,T))$, with $0\le \xi_h(t)\le 1$, and 
$\xi_h\to 1$, as $h\to0$, everywhere in $(0,T)$. If we replace $\tilde\varphi$ in the relations above for $\ell_1$ and $\ell_2$ by $\xi_h\z'_\d\bar\varphi$, with   $0\le\bar\varphi\in C^\infty(\bar U_T)$, we trivially obtain,  after taking $\d\to0$ and $h\to0$,
\begin{gather*}
\la \ell_1, \bar\varphi\ra\ge -C(k,T,\Om,\bff,a_0, \|u\|_\infty,\|\nabla_{x''}b(u)\|_{L^2},  \|\psi'',\nabla\psi''\|_\infty)\biggl( \|\bar \varphi\|_\infty\\
+\int_{\Om}(\bar\varphi(0,x)+\bar\varphi(T,x))\,dx +\int_{\Gamma_T'}\bar\varphi\,dx\,dt\biggr),\\
\la \ell_2, \bar\varphi \ra\ge -C(k,T,\Om,\bff, a_0, \|u\|_\infty,\|\nabla_{x''}b(u)\|_{L^2},  \|\psi'',\nabla\psi''\|_\infty)\biggl(\| \bar \varphi\|_\infty  \\
+\int_{\Om}(\bar\varphi(0,x)+\bar\varphi(T,x))\,dx +\int_{\Gamma_T'}\bar\varphi\,dx\,dt\biggr)-\|\psi''\|_{\infty}\int_{U_T}\bar\varphi\,d\mu_0(x)\,dt. 
\end{gather*}
Then, replacing $\bar\varphi$ by $\|\bar\varphi\|_\infty\pm \bar\varphi$ in these relations, we finally obtain
\begin{gather*}
|\la \ell_1, \bar\varphi\ra|\le C\|\bar\varphi\|_{\infty},\qquad
|\la \ell_2, \bar\varphi\ra|\le C\|\bar\varphi\|_{\infty}, 
\end{gather*}
for some $C>0$ depending on $k,T,\Om,\bff,a_0, \|u\|_\infty,\|\nabla_{x''}b(u)\|_{L^2},  \|\psi'',\nabla\psi''\|_\infty$, uniformly bounded for $k$ in bounded intervals, which implies the assertions for $\ell_1$ and $\ell_2$.  

Now, concerning $\ell_0$, we take $\psi_0''\in C_0(\Om'')$, $0\le\psi_0''\le1$, and  $\psi_J'\in C_0^\infty(B_J'')$, $B_J''\in\B''$, $J=1,\cdots,N$, such that $\sum_{J=0}^N\psi_J''(x'')=1$, for any $x''\in\Om''$.  Now, for each $\bar\varphi\in C^\infty(\bar U_T)$,  on each $B''_J$, $J\ge1$,  let  $\ell_{1,J}$ and $\ell_{2,J}$ be the distributions corresponding to $\ell_1$ and $\ell_2$, with $\psi''$ replaced by $\psi''_J$. By the definitions of $\ell_0,\ell_1,\ell_2$ and $\mu_0$, we clearly have, 
$$
\la \psi''_J\ell_0,\bar\varphi\ra \ge\la\ell_{2,J},\bar\varphi\ra -\la \ell_{1,J},\bar\varphi\ra -\int_{U_T}\psi_J''\bar \varphi\,d\mu_0(x)\,dt.
$$ 
On the other hand, by \eqref{e1.8E}, making use again of the $\Om'$-boundary layer sequence $\z'_\d(x')$, and of the sequence $\xi_h(t)$ introduced above,  we easily get
$$
\la\psi''_0\ell_0,\z'_\d\xi_h \bar\varphi\ra \ge -C\|\bar\varphi\|_\infty,
$$
from which it follows, as $h,\d\to0$, 
$$
\la\psi''_0\ell_0, \bar\varphi\ra \ge -C\|\bar\varphi\|_\infty,
$$
for some constant $C>0$, depending only on  $k,T,\Om,\bff, \|u\|_\infty, \|\psi_0'',\nabla\psi_0''\|_\infty$, uniformly bounded for $k$ in bounded intervals. Putting together these facts, we obtain as above, for all $\bar\varphi\in C^\infty(\bar U_T)$,
$$
|\la \ell_0,\bar\varphi\ra|\le C\|\bar\varphi\|_\infty,
$$
for some $C>0$ depending on  $k,T,\Om,\bff,a_0, \|u\|_\infty,\|\nabla_{x''}b(u)\|_{L^2},  \|\psi_J'',\nabla\psi_J''\|_\infty$, $J=0,\cdots,N$,  uniformly bounded for $k$ in bounded intervals, which implies the assertion for $\ell_0$.  

 \end{proof}

 \begin{remark}\label{R:1.2} Upon relabeling, reorienting  and translating the coordinates $x''$,   the hyper-surface $\po\Om''\cap B''$, in $\R^{d''}$, is the graph of a Lipschitz function,
$\bar y'' \mapsto \gamma(\bar y'')$. Let, $V_{\gamma}$ denote the hypograph of $\gamma$ in $\R^{d''}$, i.e., $y''\in\R^{d''}$ such that $y_d<\gamma(\bar y'')$.  Since in \eqref{e1.11} the $x''$ coordinates are  localized on $B''$, we may use, instead of a $\Om''$-boundary layer sequence 
$\z_\d''$, a $V_{\gamma}$-level set boundary layer sequence, still denoted $\z''_\d$,  obtained from the level set function associated with the trivial deformation for  $\po V_\gamma$,  $\Psi: [0,1]\X \po V_{\gamma}\to \bar V_{\gamma}$, given by
$\Psi(s,(\bar y'',\gamma(\bar y'')))=(\bar y'', \gamma''(\bar y'')-s)$, where $V_\gamma$ denotes the hypograph of $\gamma:\R^{d''-1}\to\R$.  We call, by abuse of nomenclature, such a boundary layer  sequence a {\em $\Om''$-canonical local boundary layer sequence}. 
\end{remark}

The following very elementary lemma will be used in the proof of the existence of an entropy solution in Section~\ref{S:4}, and  this application demonstrates how advantageous is the use of level set boundary layer sequences (see \eqref{e1.R1}), in connection with \eqref{e1.R1'}; this fact is also clear in the proof of the uniqueness in Section~\ref{S:5}. 
 
 \begin{lemma}\label{L:1.2} Let $\Bbf^*(u,v,w):=\Bbf(u,v)+\Bbf(u,w)-\Bbf(w,v)$. Then, given any $\xi\in\R^d$, $\xi^\top \Bbf^*(u,v,w)\xi \ge 0$, for any $u,v,w\in\R$.
 \end{lemma}
 
 \begin{proof} This follows immediately from the fact that the function $g_\xi(u)= \xi^\top B(u)\xi $ verifies $g_\xi'(u)\ge0$, by \eqref{e1.B'}, and, clearly, 
 \begin{multline*}
 \xi^\top \Bbf^*(u,v,w)\xi\\ =\sgn(u-v)(g_\xi(u)-g_\xi(v))+\sgn(u-w)(g_\xi(u)-g_\xi(w))-\sgn(w-v)(g_\xi(w)-g_\xi(v))\\
 =|g_\xi(u)-g_\xi(v)|+|g_\xi(u)-g_\xi(w)|-|g_\xi(w)-g_\xi(v)|\ge 0.
 \end{multline*}
 \end{proof}

 Before concluding this section, giving a brief description of the other sections in this paper, we would like to say some words about the possibility of dropping condition \eqref{e1.nondeg}. 
 Dropping  this condition seems to be a difficult problem due to the Neumann condition on the boundary $\Gamma_T'$.  The point is that this Neumann boundary condition demands the use of the strong trace property for the proof of the uniqueness of entropy solutions. This seems to preclude the use of the approach based on the uniqueness of  measure-valued solutions, as in \cite{Sz, BS},  in the purely hyperbolic case, and \cite{MV}, in the degenerate parabolic case,  both of which deal exclusively with Dirichlet boundary conditions. 
 This approach, when applicable,  gives not only uniqueness, but also existence. It has not been extended to the case of Neumann boundary conditions, the major difficulty being related to the lack of the strong trace property.  Therefore,  for the existence of solutions we need to rely on more restrictive compactness  approaches such as some  extension of the averaging lemma, as we do here, and this explains the imposition of  the non-degeneracy condition \eqref{e1.nondeg}.

 We finish this section by giving a description of the following sections. In Section~\ref{S:2}, we recall the basic facts about divergence-measure fields that will be used in this paper. 
 In Section~\ref{S:3}, we prove Theorem~\ref{T:1.2}, which is a fundamental tool for the proof of the uniqueness of entropy solutions. In Section~\ref{S:4}, we prove the existence of entropy solutions to problem \eqref{e1.1}--\eqref{e1.4}, making use of a variation of averaging lemma stated and proved in the previous section. Finally, in Section~\ref{S:5}, we conclude the proof of Theorem~\ref{T:1.1}, proving the uniqueness of entropy solutions to problem \eqref{e1.1}--\eqref{e1.4}.

\section{Divergence-Measure Fields}\label{S:2}
  
  In this section we recall some facts in the theory of divergence-measure fields that will be used in this paper.
    
 \begin{definition}\label{D:1} Let $U\subset\R^N$ be open. 
For $F\in L^p(U;\R^N)$, $1\le p\le\infty$, or $F\in\M(U;\R^N)$,
set
\begin{equation}\label{e1}
|\div F|(U):=\sup\{\,\int_U\nabla\vphi\cdot F\,\,: \,\,\vphi\in C_0^1(U),
\,\, |\vphi(x)|\le1,\ x\in U\,\}.
\end{equation}
For $1\le p\le\infty$, we say that $F$ is 
an $L^p$-divergence-measure field over $U$, i.e., $F\in\DM^p(U)$,
if $F\in L^p(U;\R^N)$ and
\begin{equation}\label{norm1}
\|F\|_{\DM^p(U)}:=\|F\|_{L^p(U;\R^N)}+|\div F|(U)<\infty.
\end{equation}
We say that $F$ is an extended divergence-measure field over $U$,
i.e., $F\in\DM^{ext}(U)$, if $F\in\M(U;\R^N)$ and
\begin{equation}\label{norm2}
\|F\|_{\DM^{ext}(U)}:=|F|(U)+|\div F|(U)<\infty.
\end{equation}
If $F\in \DM^*(U)$ for any open set $U\Subset\R^N$, 
then we say $F\in \DM^*_{loc}(\R^N)$.
\end{definition}

Here, we will be concerned only with bounded domains $U\subset\R^N$, and fields that are $L^p$ vector functions, so it will suffice to consider divergence-measure fields in 
$\DM^1(U)$.  We recall the Gauss-Green formula for general $\DM^1$-fields, first proved in \cite{CF2,CF3} and extended by Silhavy in \cite{S1}.

\begin{theorem}[Chen \& Frid \cite{CF2,CF3}, Silhav\'y \cite{S1}] \label{T:3}  If $F\in\DM^1(U)$ then there exists a linear functional $F\cdot\nu: \operatorname{Lip}(\po U)\to\R$ such that
\begin{equation}\label{e8}
F\cdot\nu(g|\po U)=\int_U \nabla g\cdot F+\int_U g\,\div F,
\end{equation}
for every $g\in \operatorname{Lip}(\R^N)\cap L^\infty(\R^N)$.  
Moreover,
\begin{equation}\label{e8'}
|F\cdot\nu(h)|\le |F|_{\DM(U)}|h|_{\Lip(\po U)},
\end{equation}
for all $h\in\Lip(\po U)$, 
 where we use the notation 
 $$
 |g|_{\Lip(C)}:= \sup_{x\in C}|g(x)|+\Lip_C(g).
 $$
 
Moreover,  let $m\,:\, \R^N\to\R$ be a nonnegative Lipschitz function with $\supp m\subset \bar U$ which is strictly positive on $U$, and for each $\ve>0$ let $L_\ve=\{ x\in U\,:\,0<m(x)<\ve\}$. Then:
\begin{enumerate}

\item[(i)] {\rm({\em cf.} \cite{CF2,CF3} and \cite{S1})} If $g\in \Lip(\R^N)\cap L^\infty(\R^N)$, we have
\begin{equation}\label{e11}
   F\cdot\nu(g|\partial U)=-\lim_{\ve\to0}\ve^{-1}\int_{L_\ve}g \,\nabla m\cdot F\,dx;
\end{equation}

\item[(ii)] {\rm({\em cf.} \cite{S1})} If 
\begin{equation}\label{e12}
\liminf_{\ve\to0} \ve^{-1} \int_{L_\ve}|\nabla m\cdot F|\,dx<\infty,
\end{equation}
then $F\cdot\nu$ is a measure over $\partial U$. 
\end{enumerate} 
A typical example of  such  $m$ is provided by $m(x)=\dist (x,\partial U)$, for $x\in U$, and $m(x)=0$, for $x\in\R^N\setminus U$. 
   
\end{theorem}

\begin{definition}\label{D:2.1}  Let $U\subset\R^{N+1}$ be  an open set. We say that  $\po U$ is a {\em   Lipschitz deformable boundary}  if the following hold:

\begin{enumerate} 

\item[(i)] For each $x\in\po U$, there exist $r>0$ and a Lipschitz mapping $\g :\R^{N}\to\R$  such that, upon relabeling, reorienting and translation,    
$$
U\cap Q(x,r) = \{\,y\in\R^{N+1}\,:\, \g(y_1,\cdots,y_{N})<y_0\,\}\cap Q(x,r),
$$
where $Q(x,r)=\{\,y\in\R^{N+1}\,:\, |y_i-x_i|\le r,\ i=1,\cdots,N+1\,\}$. We denote by $\hat \g$ the map $ \hat y\mapsto (\g(\hat y),\hat y)$, $\hat y=(y_1,\cdots,y_{N})$.

\item[(ii)]  There exists a map $\Psi:[0,1]\X\po U\to \bar U$ such that $\Psi$ is a bi-Lipschitz homeomorphism over its image and $\Psi(0,x)=x$, for all $x\in\po U$.
For $s\in[0,1]$, we denote by $\Psi_s$ the mapping from $\po U$ to $\bar U$ given by $\Psi_s(x)= \Psi(s,x)$, and set $\po U_s:=\Psi_s(\po U)$. We call such map a Lipschitz deformation for $\po U$. 

\end{enumerate}
The {\em level set function associated with the deformation} $\Psi$ is the function $h:\bar U\to [0,1]$, defined by $h(x)=s$, if $x\in\po U_s$, $0\le s\le 1$, and $h(x)=1$, if $x\in U\setminus\Psi((0,1)\X\po U)$.  
 \end{definition}
 
 \begin{definition} \label{D:2.2}  Let $U\subset\R^{N+1}$ be an open set with a  Lipschitz deformable boundary $\po U$, and $\Psi:[0,1]\X\po U \to \bar \Om$ a Lipschitz deformation.
 \begin{enumerate}
 \item  The Lipschitz deformation is said to be {\em regular over} $\Gamma\subset\po U$,  if $D\Psi_s\to \operatorname{Id}$, as $s\to0$, in $L^1(\Gamma, \H^{N})$; 
 \item The Lipschitz deformation is said to be {\em strongly regular over} $\Gamma\subset\po U$ if it is regular over $\Gamma$ and  $J[\Psi_s]\to 1$ in $\Lip(\Gamma)$, as $s\to0$, that is,  given any Lipschitz diffeomorphism $\hat \g: \Om\subset\R^{N}\to \Gamma$, we have  $D\Psi_s\circ\hat\g\to D\hat\g$ in $L^1(\Om)$, as $s\to0$, and  $J[\Psi_s\circ\hat\g]/ J[\hat\g]\to 1$ in $\Lip(\Om)$, as $s\to0$ . Here, for a Lipschitz function $\a:\R^k\to\R^m$ we denote by $J[\a]$ the Jacobian of the map $\a$ (see, e.g., \cite{Fe}). Observe that we do not need to require more regularity on $\Gamma$, it suffices that $J[\Psi_s]\in\Lip(\Gamma)$.  
 \end{enumerate}
  \end{definition}

The following two results have been proved in \cite{Fr}. We include their proofs here for the convenience of the reader. 

\begin{theorem}[{\em cf.} \cite{Fr}]\label{T:2.3} Let $U\subset\R^{N+1}$ be a bounded open set with a deformable Lipschitz boundary and  $F\in\DM^1(U)$. Let $\Psi:\po U\X[0,1]\to\bar U$ be a Lipschitz deformation of $\po U$. Then, for almost all $s\in[0,1]$, and all $\phi\in C_0^\infty(\R^{N+1})$, 
\begin{equation}\label{e2.2'}
\int_{U_s}\phi\,\div F = \int_{\po U_s}\phi(\om) F(\om)\cdot\nu_s(\om)\,d\H^{N}(\om)-\int_{U_s}F(x)\cdot\nabla\phi(x)\,dx,
\end{equation}
where $\nu_s$ is the unit outward normal field defined $\H^{N}$-almost everywhere in $\po U_s$, and $U_s$ is the open subset of $U$ bounded by $\po U_s$.
\end{theorem}

\begin{proof} For $\phi\in C_0^\infty(\R^{N+1})$, let  
$$
\z_\phi(s)=\int_{\po U_s} \phi(\om) F(\om)\cdot\nu_s (\om)\,d\H^{N}(\om), \quad s\in[0,1],
$$ 
where $\nu_s$ is as in the statement. Let $s_0\in[0,1]$ be a Lebesgue point for $\z_\phi$, for $\phi$ in a countable dense set in $C_0^\infty(\R^{N+1})$.  For $\d>0$ sufficiently small,
let $g_\d:\R\to\R$ be defined as
$$
g_\d(s)=\begin{cases} 0, & s<s_0-\d,\\ \frac{s-s_0+\d}{2\d}, &s_0-\d\le s\le s_0+\d,\\ 1, &s>s_0+\d.\end{cases}
$$
Set $\psi_\d=g_\d\circ h\, \phi$, where $h$ is the level set function associated to the Lipschitz deformation $\Psi$. By the Gauss-Green formula, we have
\begin{align*}
0=&\int_{U} F\cdot\nabla\psi_\d\,dx+\int_{U}\psi_\d\, \div F\\
&= \int_{U}\phi g_\d'(h(x)) F\cdot\nabla h\,dx+\int_{U} g_\d(h(x)) F\cdot\nabla\phi\,dx+\int_{U}\psi_\d\,\div F ,
\end{align*}
which gives, by the coarea formula,
$$
0=-\frac1{2\d}\int_{s_0-\d}^{s_0+\d}\int_{\po U_s}\phi F\cdot\nu_{s}\,d\H^{N}(\om)\,ds+\int_{U} g_\d(h(x)) F\cdot\nabla\phi\,dx+\int_{U}\psi_\d\,\div F .
$$
Letting $\d\to0$, we obtain \eqref{e2.2'} for $s=s_0$, where $s_0$ is an arbitrary Lebesgue point of $\z_\phi$, for $\phi$ in a countable dense subset of $C_0^\infty(\R^{N+1})$,
and, so, \eqref{e2.2'} holds for almost all $s\in[0,1]$ as it was to be proved.

\end{proof}

\begin{theorem}[{\em cf.} \cite{Fr}] \label{T:2.3'} Let $F\in\DM^1(U)$, where $U\subset\R^{N+1}$ is a bounded open set with a Lipschitz deformable boundary and Lipschitz   deformation $\Psi:[0,1]\X\po U\to\bar U$.
Denoting by $F\cdot\nu|_{\po U}$ the continuous linear functional $\Lip(\po U)\to\R$ given by the normal trace of $F$ at $\po U$, we have the formula
\begin{equation}\label{e2.3'}
F\cdot\nu|_{\po U}=\esslim_{s\to0} F\circ\Psi_s(\cdot)\cdot\nu_s(\Psi_s(\cdot))J[\Psi_s] ,
\end{equation}
with equality in the sense of $(\Lip(\po U))^*$, where on the right-hand side the  functionals are given by ordinary functions in $L^1(\po U)$. In particular, if $\Psi$ is strongly regular  over $\Gamma\subset\po U$   then, for all $\varphi\in \Lip(\po U)$ with $\supp\varphi\subset\Gamma$, we have
 \begin{equation}\label{e2.3''}
\la F\cdot\nu|_{\po U},\varphi\ra=\esslim_{s\to0} \int_{\Gamma} F\circ\Psi_s(\om)\cdot\nu_s(\Psi_s(\om))\varphi(\om)\,d\H^{N}(\om).
\end{equation}
\end{theorem}

\begin{proof}

{}From Theorem~\ref{T:2.3} and the Gauss-Green formula \eqref{e8}, when $F\in\DM^1(\Om)$, it follows that, for any   $g\in\Lip(\R^{N+1})\cap L^\infty(\R^{N+1})$, we have the following formula for the normal trace functional $F\cdot\nu:\Lip(\po U)\to\R$,
\begin{equation}\label{e2.3'''}
\la F\cdot\nu,g|\po U\ra=\esslim_{s\to0}\int_{\po U_s} g F(\om)\cdot\nu(\om)\,d\H^{N}(\om),
\end{equation}
where the limit on the right-hand side exists by applying dominated convergence to the other two terms in \eqref{e2.2'}. Therefore, for any $\phi\in\Lip(\po U)$, we have
\begin{equation*}
\la F\cdot\nu,\phi \ra=\esslim_{s\to0}\int_{\po U_s} \phi\circ\Psi_s^{-1}(\om) F(\om)\cdot\nu_s(\om)\,d\H^{N}(\om),
\end{equation*}
or, by using the area formula,
\begin{align*}
\la F\cdot\nu,\phi \ra&=\esslim_{s\to0}\int_{\po U} \phi(\om) F\circ\Psi_s(\om)\cdot\nu_s(\Psi_s(\om)) J[\Psi_s]\,d\H^{N}(\om)\\
                                   &=\esslim_{s\to0}\int_{\po U} \phi(\om) F\circ\Psi_s(\om)\cdot\nu(\Psi_s(\om)) \,d\H^{N}(\om)\\
                                   &+\esslim_{s\to0}\int_{\po U} \phi(\om)\frac{(J[\Psi_s]-1)}{J[\Psi_s]} F\circ\Psi_s(\om)\cdot\nu(\Psi_s(\om))J[\Psi_s] \,d\H^{N}(\om)\\
                                   &=\esslim_{s\to0}\int_{\po U} \phi(\om) F\circ\Psi_s(\om)\cdot\nu(\Psi_s(\om)) \,d\H^{N}(\om),
\end{align*}  
where the first equality is \eqref{e2.3'}, and the last equality holds, through \eqref{e2.3'}, if $\Psi$ is a strongly regular Lipschitz deformation in the sense of Definition~\ref{D:2.2}, which completes the proof.  
\end{proof}

   \medskip
\begin{remark}\label{R:2.1}
Let $U$ be endowed with a Lipschitz boundary  and $\Gamma\subset\po U$ be an open  piece of $\po U$. Let $\om\in\Gamma$ and $\CR_\om:\R^{N+1}\to\R^{N+1}$ be a rigid motion in $\R^{N+1}$ such that, for some Lipschitz function  $\g:\R^{N}\to\R$, denoting $y=\CR_\om x$, $\hat y=(y_1,\cdots,y_{N})$,  and defining $\hat\g(\hat y):=(\g(\hat y),\hat y)$,  we have  that $\hat\g(\Om)=\Gamma$, for some open set  $\Om\subset\R^{N}$. Let also $\tilde U=\{y\in\R^{N+1}\,:\, \g(\hat y)<y_0\,\}$ and suppose $U\cap \tilde U\ne\emptyset$. If $F\in\DM^1(U)\cap\DM^1(\tilde U)$, it is immediate to check, using the Gauss-Green formula, that $F\cdot\nu|_{\po U}$ and $F\cdot\nu|_{\po\tilde U}$ coincide over $\Gamma$, that is,
$$
\la F\cdot\nu|_{\po U},\varphi\ra=\la F\cdot\nu|_{\po\tilde U},\varphi\ra,
$$
for all $\varphi\in \Lip_c(\Gamma)$, where $\Lip_c(\Gamma)$ denotes the subspace of  functions in $\Lip(\Gamma)$ with compact support in $\Gamma$.  Also recall that, to define 
$F\cdot\nu|_{\po U}$,    it is not necessary that $\po U$ be Lipschitz deformable.  In such cases, restricted to functions $\varphi\in \Lip_c(\Gamma)$ we will always view $\la F\cdot \nu|_{\po U},\varphi\ra$ as obtained, after translation, relabeling and reorienting coordinates, 
 through \eqref{e2.3''} by using the {\em canonical  deformation} of $\po \tilde U$ , defined as  $(s, (\gamma(\hat y),\hat y))\mapsto (\gamma(\hat y)+s, \hat y)$, evidently strongly regular over $\Gamma$, which is legitimate for $\tilde U$; we call that a {\em local canonical deformation of $\po U$}.

\end{remark}

\begin{remark}\label{R:2.2} Since the normal trace $F\cdot\nu|_{\po U}$, of a divergence measure field over an open set $U$, restricted to some open piece $\Gamma\subset \po U$, 
does not depend on $U$, but just on $\Gamma$, for the purpose of defining $F\cdot\nu$ over $\Gamma$, we may refer to a deformation  $\Psi: [0,1]\X\Gamma\to\bar U$, defined just on $[0,1]\X\Gamma$, which may be the restriction over $[0,1]\X\Gamma$ of a deformation $\tilde\Psi:[0,1]\X\po\tilde U\to\overline{\tilde U}$ such that $\Gamma\subset\po\tilde U$,
$U\cap\tilde U\ne\emptyset$ and $F$ may be extended somehow from $U\cap\tilde U$ to $\tilde U$ so as to be viewed as  a divergence-measure field over $\tilde U$.  
\end{remark}

\begin{remark}\label{R:2.3} If $\Om'=(\a_1,\b_1)\X\cdots\X(\a_{d'},\b_{d'})$ and we assume, without loss of generality, that the origin of $\R^{d'}$ is in the center of  $\Om'$, then $\Psi': \po\Om'\X[0,1]\to \overline{\Om'}$, given by $\Psi'(x,s)=(1-\ve s)x$, for any $0<\ve<1$, is a strongly regular deformation of $\po\Om'$ as it is trivial to check. More generally, let $\Om'$ be a {\em starlike polyhedric domain}, that is, a bounded domain whose boundary is piecewise flat and there is a point $x_0$ in its interior such that, for each point $x\in\po\Om'$, the segment $(1-\theta)x+\theta x_0$, $0<\theta <1$, is entirely contained in $\Om'$. Then again, identifying $x_0\equiv0$, for simplicity, $\Psi': \po\Om'\X[0,1]\to \overline{\Om'}$, given by $\Psi'(x,s)=(1-\ve s)x$,    for any $0<\ve<1$, is a strongly regular deformation of $\po\Om'$, which is also easy to check.  

\end{remark}

\section{Strong trace property}\label{S:3}

 In this section, we prove  a strong trace property for the entropy solution of  equation \eqref{e4.1}, which  extends the strong trace property first established by Vasseur~\cite{Va}, for scalar conservation laws, when flux functions are almost everywhere 
 non-degenerate, which corresponds to \eqref{e1.nondeg} in the purely hyperbolic case when $d'=d$,  so $d''=0$ and $b\equiv0$.  In this connection, we recall that Panov~\cite{Pv}  obtained an extension of  the result in \cite{Va}, still for the hyperbolic case, dropping completely the  non-degeneracy restriction over the flux function (see also \cite{KV}, concerning the same question in the one-dimensional case). We also mention that in \cite{Kw} we find a first attempt to extend the trace property for  degenerate parabolic equations in the isotropic case. 
 
 Before passing to the statement and proof of the strong trace property, we would like to give a very brief overview of what the result means and the basic strategy of the proof. We are going to prove, in particular, that, given a strongly  regular deformation for the hyperbolic boundary, based on a deformation for  $\po\Om'$, the entropy solution of problem  \eqref{e1.1}--\eqref{e1.4}, restricted to the corresponding family of surfaces,  has a strong limit, in the $L^1$ sense, when approaching the hyperbolic boundary.  The proof follows the steps of the proof in \cite{Va}, with the necessary adaptations for the parabolic equation \eqref{e1.1}.  The proof starts by introducing the kinetic formulation for the entropy inequality for \eqref{e1.1}. 
 Then, we apply Theorem~\ref{T:2.3'} in order to obtain a well defined weak trace for the kinetic $\chi$-function. The fact that $\chi$-functions assume values in $\{-1,0,1\}$ implies that, if the weak trace is a $\chi$-function, then the convergence is strong and the weak trace is, in fact, a strong trace, from where it easily follows the strong trace property for the entropy solution. To prove that the weak trace is a $\chi$-function, use is made of the localization (blow-up) method introduced in \cite{Va}. The idea is to fix a point in the boundary, arbitrarily chosen from a set of full measure, introduce new variables around this point, define a sequence of $\chi$-functions from the kinetic $\chi$- function by scaling these new variables by means of a parameter $\ve$, making $\ve\to0$, and applying a suitable version of the averaging lemma ({\em cf.} \cite{PB}) to show that the scaling sequence of $\chi$-functions converges strongly in $L_\loc^1$, and so its limit is a $\chi$-function. Since the initial functions for the scaling sequence are given by the corresponding scaling of the weak trace, which strongly converge to the value at the fixed point, the final step is then to prove, as in \cite{Va},  that the initial function of the limit is the limit of the initial functions.    Now we pass to the rigorous statement and its proof.

 \begin{theorem}\label{T:5.1}  Assume that\eqref{e1.TT} holds, and let  $u(t,x)\in L^\infty(U_T)$ satisfy $\nabla_{x''}b(u)\in L^2(U_T)$ and,  for all $0\le \varphi\in C_0^\infty(U_T)$ and $k\in\R$,
 \begin{equation}\label{e5.3}
\int_{U_T}\{|u-k|\po_t\varphi-K_{x''}(u,k)\cdot\nabla \varphi\}\,dx\,dt\ge0.
\end{equation}
Then, there exists $u^\tau\in L^\infty(\Gamma_T')$  such that, for any  deformation of $\po\Om'$,  $\Psi': [0,1]\X \po\Om'\to \bar\Om'$, strongly regular over $\po\Om'$,  if  $\Psi: [0,1]\X\Gamma'_T \to \bar U_T$, is defined by $\Psi(s,t,x',x'')=(t,\Psi'(s,x'),x'')$, we have $u(\Psi(s,t,x))\to u^\tau(t,x)$, as $s\to0$, in $L^1(\Gamma_T')$.
\end{theorem} 
  
\begin{proof} Before starting the proof, we observe that although \eqref{e5.3} does not imply that $(|u-k|,-K_{x''}(u,k))\in \DM^2(U_T)$, it is nevertheless easy to verify, using the arguments in the proof of Lemma~\ref{L:1.1}, that for any Lipschitz deformation of $\Om''$, $\Om''_s$, $0\le s\le 1$, we have $(|u-k|,-K_{x''}(u,k))\in \DM^2(U_{sT})$,
where $U_{sT}=(0,T)\X\Om'\X\Om_s''$, for $0<s\le1$. Therefore, we may apply the following proof to $U_{sT}$ for any $s>0$, and then get the stated result for $U_T$.  

We divide the proof into six steps. 

Step \#1.  It is well known (see, e.g., \cite{CP}) that if $u\in L^\infty(U_T)$ satisfies \eqref{e5.3}, then the function
 $$
 f(t,x,\xi)=\chi(\xi; u(x,t)),\quad \text{where}\quad \chi(\xi;u):=\begin{cases} -1, \quad u\le \xi< 0,\\ 1,\quad  0<\xi \le u,\\ 0,\quad |\xi|>|u|,\end{cases}
 $$
 satisfies
 \begin{equation} \label{e5.4}
 \po_t f+\bfa(\xi)\cdot\nabla f-B'(\xi):\nabla_{x''}^2 f=\po_{\xi} m,
 \end{equation}
 in the sense of distributions in $\DD'(U_T\X(-L,L))$, with $\bfa(\xi)=\bff'(\xi)$, for some $m\in\M_+(U_T\X(-L,L))$, where $\M_+(U_T\X(-L,L))$ denotes the space of non-negative Radon measures on $U_T\X(-L,L)$.  Indeed, choosing $k=\pm \|u\|_\infty$ in \eqref{e5.3}, as usual, we deduce that $u$ satisfies \eqref{e1.1} in the sense of distributions in $U_T$. Together with \eqref{e5.3}, this implies that for any convex function $\eta:\R\to\R$, we have, in the sense of the distributions in $U_T$, 
 \begin{equation}\label{e5.5}
 \eta(u)_t+\nabla\cdot \bff_{\eta}(u)-\sum_{i,j=d'+1}^d\po_{x_ix_j}^2 b_{ij\eta}(u)=-m_{\eta}(t,x),
 \end{equation}
 where $\bff_\eta'(u)=\eta'(u)\bff'(u)$, $b_{ij\eta}'(u)=\eta'(u)b_{ij}'(u)$, and $m_\eta\in\M_+(U_T)$.  Since, we have
 \begin{gather*}
 \eta(u)=\int_{\R}\eta'(\xi)\chi(\xi;u)\,d\xi,\quad \bff_{\eta}(u)=\int_{\R}\eta'(\xi)\bff'(\xi)\chi(\xi;u)\,d\xi,\\
  b_{ij\eta}(u)=\int_{\R}\eta'(\xi)b_{ij}'(\xi)\chi(\xi;u)\,d\xi,
 \end{gather*}
(assuming $\eta(0)=b_{ij\eta}(0)=0$, $\bff_\eta(0)=0$)  and, for each fixed $\xi\in\R$,  $\eta_\xi(u)=(\xi-u)_+-(\xi)_+$, with $(\s)_+=\max\{\s,0\}$,  is convex,  writing \eqref{e5.5} for $\eta=\eta_\xi$,  deriving with respect to $\xi$, multiplying by $\eta'(\xi)$, for an arbitrary smooth convex function $\eta$, and integrating the result with respect to $\xi$, we conclude that $m_\eta(t,x)=\int_{\R}\eta''(\xi)\,m(t,x,\xi)$, with $m(t,x,\xi)=m_{\eta_\xi}(t,x)$, and $f(t,x,\xi)=\chi(\xi;u(t,x))$ satisfies \eqref{e5.4},  in $\DD'(U_T\X(-L,L))$.   

Step \#2. Following \cite{Va}, since $\po\Om'$ is a Lipschitz boundary, we  fix an open subset of $\po\Om'$, $\S_\a'$, parameterized by $(\g_\a'(\hat y'),\hat y')$, $\hat y'\in (-r_\a,r_\a)^{d'-1}$, for a certain Lipschitz function $\g_\a':\R^{d'-1}\to\R$, and we denote by $\CR_\a'$ the isometry on $\R^{d'}$ representing the translation, relabeling and reorienting necessary to parameterize $\S_\a'$ as the graph of $\g_\a'$ over $(-r_\a,r_\a)^{d'-1}$. Let $\Psi'$ be a deformation of $\po\Om'$ strongly regular over $\S_\a'$, and let $\Psi:[0,1]\X\Gamma_T'\to\bar U_T$ be defined by 
\begin{equation}\label{e5.Psi}
\Psi(s,t,x',x'')=(t,\Psi'(s,x'),x'').
\end{equation}
We set  $\CR_\a(t,x',x''):=(\CR_\a'(x'),t,x'')=(y',y''')$, $\g_\a(\hat y):=\g_\a'(\hat y')$,  $\hat y=(\hat y', y''')\in \R^d=\R^{d'-1}\X\R^{d''+1}$, where we set $y'''=(t,x'')$, and, for $\hat y\in (-r_\a,r_\a)^d$, define
\begin{equation}
\psi(s,\hat y)=\CR_\a\Psi(s,\CR_\a^{-1}( \g_\a(\hat y),\hat y)).\label{e5.Psi'}
\end{equation}
Writing \eqref{e5.4} in the $y=\CR_\a(t,x)$ coordinates, we get 
\begin{equation}\label{e5.4'}
a(\xi)\cdot\nabla_y f-B'(\xi):\nabla_{\hat y''}^2 f=\po_{\xi} m,
 \end{equation}
 where $\hat y''=x''$, $a(\xi)=(a^0(\xi),\hat a(\xi)):=A_\a'(\xi)$ and 
 $$
 A_\a(\xi)=\CR_\a(\xi,\bff(\xi))=(\CR_\a'(\pi_{d'}(\bff(\xi))), \xi, \pi_{d''}(\bff(\xi))),
 $$ 
 that is, $a(\xi)=\CR_\a(1,\bfa(\xi))=(\CR_\a'(\pi_{d'}(\bfa(\xi))), 1, \pi_{d''}(\bfa(\xi)))$, where $\pi_{d''}(v_1,\cdots,v_d)=(v_{d'+1},\cdots,v_d)$. 
 Set
\begin{equation}
f_\psi(s,\hat y,\xi)=f(\psi(s,\hat y),\xi).\label{e5.Psi''}
\end{equation}
Now we use Theorem~\ref{T:2.3'}, more specifically relation \eqref{e2.3''}, to prove the existence of a weak trace for $f_\psi$, which is the analogue of proposition~1, in \cite{Va}.

{\em Claim \#1}. Let $f$ be a solution of \eqref{e5.4} in 
\begin{equation}\label{e.orient}
\Om_\a:=\{y\in(-r_\a,r_\a)^{d+1}\,:\,y_0>\g_\a(y_1,\cdots,y_d)\}
\end{equation}
 with $\bfa(\xi)$ and $b(\xi)$ satisfying \eqref{e1.nondeg}, respectively. Then there is a unique $f^\tau\in L^\infty((-r_\a,r_\a)^d\X(-L,L))$ such that, for $\Psi$ given by \eqref{e5.Psi}, where $\Psi'$ is any Lipschitz deformation of $\po\Om'$ strongly regular over $\S_\a'$, we have
\begin{equation}\label{e5.5'} 
\esslim_{s\to0} f_\psi(s,\cdot,\cdot)=f^\tau\quad\text{in $H^{-1}((-r_\a,r_\a)^d\X(-L,L))$}.
\end{equation}

The proof of this claim reduces to the same proof of proposition~1, in \cite{Va}, but here we have to use Theorem~\ref{T:2.3'}, instead of the corresponding result for $\DM^\infty$-fields in \cite{CF1}, and we have to use the fact that $F(\Psi_s(\om))\cdot\nu_s$ is uniformly bounded in $L^\infty(\Gamma_T')$, since the normal to 
$\Psi_s(\Gamma_T')$ is orthogonal to the space of the coordinates $(t,x'')$. 

Lemma~1 in \cite{Va} establishes the equivalence, for a sequence of   $\chi$-functions $f_n(z,\xi)=\chi(u_n(z),\xi)$ converging weakly* to $f$ in $L^\infty(\CO\X(-L,L))$, of the assertions: (i) $f_n$ converges strongly to $f$ in $L^1_\loc(\CO\X(-L,L))$; (ii) $u_n$ converges strongly  to $u(\cdot)=\int_{-L}^L f(\cdot,\xi)\,d\xi$; (iii) $f$ is a $\chi$-function. Due to this result, we have the following analogue of proposition~2 of \cite{Va}:

{\em Claim \#2}. The function $f^\tau$ is a $\chi$-function if and only if $f^\tau$ is a strong trace, namely, for every   $\Psi$ given by \eqref{e5.Psi}, where $\Psi'$ is any Lipschitz deformation of $\po\Om'$ strongly regular over $\S_\a'$, 
$$
\esslim_{s\to0} f_\psi(s,\cdot,\cdot)=f^\tau\quad \text{in $L^1((-r_\a,r_\a)^d\X(-L,L))$}.
$$ 

Step \#3. Next, we follow the so called localization method in section~4 of \cite{Va}, whose goal is the proof that $f^\tau(\hat y, \cdot)$ is a $\chi$-function for almost every $\hat y\in(-r_\a,r_\a)^d$, so proving the strong convergence of $f_\psi(s,\cdot,\cdot)$, as $s\to0$. For simplicity, from now on, we fix the deformation $\Psi'$ of $\po\Om'$ as given over $\S_\a$ by
$\psi'(s,\hat y'):=\CR_\a'\Psi'(s,\CR_\a'^{-1}(\g'_\a(\hat y'), \hat y'))= (s+\g_\a'(\hat y'),\hat y')$, and so, for $\psi$ given by \eqref{e5.Psi'}, we have $\psi(s,\hat y)=(s+\g_\a(\hat y),\hat y)$. We then identify $s=y_0$ and define
$$
\tilde f(y,\xi)=f_\psi(y_0,\hat y,\xi).
$$
We have $\psi(y)\in\Om_\a$ if and only if $y\in(0,r_\a)\X(-r_\a,r_\a)^d$. 
 We easily check that $\tilde f$ is a solution of 
\begin{equation}\label{e5.6}
\tilde a^0(\hat y,\xi)\po_{y_0}\tilde f + \hat a(\xi)\cdot\nabla_{\hat y}\tilde f-B'(\xi):\nabla_{\hat y''}^2 \tilde f=\po_\xi \tilde m,
\end{equation}
where 
\begin{equation} \label{e5.7}
\begin{aligned}
\tilde a^0(\hat y,\xi)&=a^0(\xi)-\nabla\g_\a(\hat y)\cdot \hat a(\xi)\\
&=\lambda(\hat y) \left(a^0(\xi)\nu_0+\pi_{d'}(\hat a)(\xi)\cdot  \pi_{d'}(\hat\nu(\hat y))\right),
\end{aligned}
\end{equation}
with $\l(\hat y)> 0$, and $\tilde m(y,\xi)= m(\psi(y),\xi)$, where we use the fact that $\nabla\g_\a(\hat y)$ is orthogonal to the space of the $y''$ coordinates and $\nu=(\nu_0,\hat \nu)$ is the  outward unit normal to the graph $(\g_\a(\hat y),\hat y)$, $\hat y\in(-r_\a,r_\a)^d$. We remark, in particular, that $\tilde a^0(\hat y,\xi)\ne0$, for a.e.\ $\xi$, for each $\hat y$, due to the non-degeneracy condition \eqref{e1.nondeg}, which, written  for $a(\xi)=\CR_\a(1,\bfa(\xi))$, takes the form
$$
\LL^1\{\xi\in\R\,:\, a^0(\xi)\z_0+\pi_{d'}(\hat a)(\xi)\cdot \z=0\}=0, \quad\text{for any $(\z_0,\z)\in\R^{d'+1}$, $\z_0^2+|\z|^2=1$}.
$$

At this point we observe that \eqref{e.orient} involves a choice of orientation for the $y_0$-axis. On the other hand, if we had chosen the opposite orientation 
with $y_0<\g_\a(\hat y)$, then the deformation now would read $\psi(\widehat{y},s) = (\gamma_\a(\widehat{y}) - s, \widehat{y})$ and, with the identification $y_0=-s$, the coefficient of the derivative $\po/\po y_0$ would still be 
given by the first equation in \eqref{e5.7}.

Next, following \cite{Va}, we have the following two claims which are the analogues of lemma~2 and lemma~3 of \cite{Va}, whose respective proofs are identical to those of the latter.  

{\em Claim \#3}.  There exists a sequence $\ve_n$ which converges to 0 and a set $\E\subset (-r_\a,r_\a)^d$ with $\LL((-r_\a,r_\a)^d\setminus\E)=0$ such that for $\hat y\in\E$ and for every $R>0$
$$
\lim_{n\to\infty}\frac1{\ve_n^{d'+\frac{d''}2}}\tilde m\left((0,R\ve_n)\X(\hat y+\Lambda(\ve_n)(-R,R)^{d}))\X(-L,L)\right)=0,
$$
where $\Lambda(\ve_n)(\hat z)=(\ve_n\hat z', \ve_n^{1/2}\hat z'')$, for all $\hat z=(\hat z',\hat z'')\in\R^d=\R^{d'}\X\R^{d''}$, and so 
$$
\Lambda(\ve_n)(-R,R)^d=(-R\ve_n,R\ve_n)^{d'}\X(-R\ve_n^{1/2}, R\ve_n^{1/2})^{d''}.
$$

{\em Claim \#4}. There exists a subsequence still denoted $\ve_n$ and a subset $\E'$ of $(-r_\a,r_\a)^d$ such that $\E'\subset\E$, $\LL((-r_\a,r_\a)^d\setminus \E')=0$, and for every $\hat y\in\E'$ and every $R>0$,
\begin{align*}
&\lim_{\ve_n\to0}\int_{-L}^L\int_{(-R,R)^d}|f^\tau(\hat y,\xi)-f^\tau(\hat y+\Lambda(\ve_n)\underline{\hat y},\xi)|\,d\underline{\hat y}\,d\xi=0,\\
&\lim_{\ve_n\to0}\int_{-L}^L\int_{(-R,R)^d}|\tilde a^0(\hat y,\xi)-\tilde a^0(\hat y+\Lambda(\ve_n)\underline{\hat y},\xi)|\,d\underline{\hat y}\,d\xi=0.
\end{align*}

Step \#4.  Keeping following the lines of \cite{Va}, we define
$$
\Om_\a^\ve=(0,r_\a/\ve)\X\left(\Lambda(\ve)^{-1}(-r_\a,r_\a)^d\right),
$$
where $\Lambda(\ve)^{-1}(\hat z)=(\ve^{-1} \hat z',\ve^{-1/2}\hat z'')$ and so $\Lambda(\ve)^{-1}(-r_\a,r_\a)^d=(-r_\a/\ve,r_\a/\ve)^{d'}\X(-r_\a/\ve^{1/2},r_\a/\ve^{1/2})^{d''}$.  The final goal is to show that for every $\hat y\in\E'$, $f^\tau(\hat y,\cdot)$ is a $\chi$-function. We thus fix $\hat y\in\E'$, and denote $y=(0,\hat y)$ the associated point on $\Gamma_\a:=\{(\g_\a(\hat y),\hat y)\,:\, \hat y\in(-r_\a,r_\a)^d\}$. We then rescale the function $\tilde f$, introducing a new function $\tilde f_\ve$ which depends on a new variable $\underline{y}=(\underline y_0,\underline{\hat y}) \in\Om_\a^\ve$, defined by
\begin{equation}\label{e5.8}
\tilde f_\ve(\underline y,\xi)=\tilde f(y+(\ve \underline y_0, \Lambda(\ve) \underline{\hat y}),\xi),
\end{equation}
where we omit the dependence of $\tilde f_\ve$ on $y$ since it is going to be fixed until the end of the proof. As in \cite{Va}, we have the relation
\begin{equation}\label{e5.9}
\tilde f_\ve(0,\underline{\hat y},\xi)=f^\tau(\hat y+\Lambda(\ve)\underline{\hat y}, \xi).
\end{equation}
The point is then to study the limit of $\tilde f_\ve$ as $\ve\to0$, in order to gain knowledge on $f^\tau(\hat y,\cdot)$. We set
\begin{equation}\label{e5.10}
\tilde a_\ve^0(\underline{\hat y},\xi)=\tilde a^0(\hat y+\Lambda(\ve)\underline{\hat y},\xi),
\end{equation}
and, from \eqref{e5.6}, we get
\begin{equation}\label{e5.11}
\tilde a_\ve^0(\un{\hat y},\xi)\po_{\un{y}_0} \tilde f_\ve+\hat a'(\xi)\cdot\nabla_{\un{\hat y}'}\tilde f_\ve+\ve^{1/2}\,\hat a''(\xi)\cdot\nabla_{\un{\hat y}''}\tilde f_\ve-B'(\xi)\cdot\nabla_{\un{\hat y}''}^2\tilde f_\ve=\po_\xi\tilde m_\ve,
\end{equation}
where $\tilde m_\ve$ is the nonnegative measure defined for every real $R_1^i<R_2^i$, $L_1<L_2$ by
\begin{multline}\label{e5.12}
\tilde m_\ve\left(\underset{0\le i\le d}{\Pi}[R_1^i,R_2^i]\X[L_1,L_2]\right)\\
=\frac{1}{\ve^{d'+\frac{d''}2}}\tilde m\left([\ve R_1^0, \ve R_2^0]\X(\hat y+\Lambda(\ve)\underset{1\le i\le d}{\Pi}[R_1^i,R_2^i])\X[L_1,L_2]\right).
\end{multline}
We may rewrite \eqref{e5.11}  in the following more suitable form
\begin{multline}\label{e5.11'}
\tilde a^0(\hat y ,\xi)\po_{\un{y}_0} \tilde f_\ve+\hat a'(\xi)\cdot\nabla_{\un{\hat y}'}\tilde f_\ve-B'(\xi):\nabla_{\un{\hat y}''}^2\tilde f_\ve=\\
-\ve^{1/2}\,\hat a''(\xi)\cdot\nabla_{\un{\hat y}''}\tilde f_\ve + \po_{\un{y}_0} \left((\tilde a^0(\hat y,\xi)- \tilde a_\ve^0(\un{\hat y},\xi)) \tilde f_\ve\right)+
\po_\xi\tilde m_\ve.
\end{multline}
By Claims~\#3 and~\#4 the distributions on the right-hand side of \eqref{e5.11'}  converge to 0 in the sense of the distributions on $(0,\infty)\X\R^{d'}\X\R^{d''}\X (-L,L)$. since $\tilde f_\ve$ is uniformly bounded, by passing to a subsequence if necessary 
we have that there exists $\tilde f_\infty \in L^\infty((0,\infty)\X\R^{d'}\X\R^{d''}\X(-L,L))$ such that $\tilde f_\ve\wto \tilde f_\infty$ in the sense of the distributions on $(0,\infty)\X\R^{d'}\X\R^{d''}\X(-L,L)$. We then see that $\tilde f_\infty$ satisfies
 \begin{equation}\label{e5.13}
 \tilde a^0(\hat y,\xi)\po_{\un{y}_0}\tilde f_\infty+\hat a'(\xi)\cdot\nabla_{\un{\hat y}'}\tilde f_\infty-B'(\xi):\nabla_{\un{y}''}^2\tilde f_\infty=0.
 \end{equation}

 We then have the following claim corresponding to proposition~3 of \cite{Va}.

 {\em Claim \#5}. There exists a subsequence $\ve_n$ going to 0 such that $\tilde f_{\ve_n}$ converges  strongly to $\tilde f_\infty$ in $L_\loc^1(\R_+\X\R^d\X(-L,L))$ and so $\tilde f_\infty$ is a $\chi$-function and satisfies \eqref{e5.13}.

 In order to address the convergence asserted in this claim, the most suitable result so far in the literature is the averaging lemma ~2.3 by Tadmor and Tao in \cite{TT}. Actually, in \cite{LPT} a very general result is stated (theorem~C, in the appendix of \cite{LPT}) which would be totally adequate  for the case we are interested in here, however no proof is provided and such a general result seems still to be lacking a proof. However, as we will now show, the averaging lemma~2.3 in \cite{TT} is enough if $\Om'$ is a starlike polyhedric domain (see Remark~\ref{R:2.3}).  
 
 More specifically, let us first localize the equation \eqref{e5.11'} by multiplying it by a bump function $\Phi(\ul y,\xi)=\Phi_1(\ul y) \Phi_2(\xi)$, with $\Phi_1\in C_c^\infty((0,\infty)\X\R^d)$, $\Phi_2\in C_c^\infty((-L,L))$, and  let us denote 
 $\tilde f_\ve^\Phi=\Phi \tilde f_\ve$.  We then get te following equation for $\tilde f_\ve^\Phi$,
 \begin{multline}\label{e5.11''}
\tilde a^0(\hat y ,\xi)\po_{\un{y}_0} \tilde f_\ve^\Phi+\hat a'(\xi)\cdot\nabla_{\un{\hat y}'}\tilde f_\ve^\Phi -B'(\xi):\nabla_{\un{\hat y}''}^2\tilde f_\ve^\Phi=\\
-\ve^{1/2} \Phi \hat a''(\xi)\cdot\nabla_{\un{\hat y}''}\tilde f_\ve  + \po_{\un{y}_0} \left((\tilde a^0(\hat y,\xi)- \tilde a_\ve^0(\un{\hat y},\xi)) \tilde f_\ve^\Phi \right)+
\po_\xi(\Phi \tilde m_\ve)\\
+ \tilde f_\ve \tilde a^0(\hat y,\xi)  \po_{\ul{y}_0}\Phi+ \tilde f_\ve \hat a'(\xi)\cdot\nabla_{\un{\hat y}'}\Phi-\tilde f_\ve B'(\xi)\cdot\nabla_{\un{\hat y}''}^2\Phi+ \s(\xi)\nabla_{\ul{y}''}\Phi\otimes \s(\xi)\nabla_{\ul{y}''}\tilde f_\ve\\
-  \left((\tilde a^0(\hat y,\xi)- \tilde a_\ve^0(\un{\hat y},\xi)) \tilde f_\ve \right)\po_{\un{y}_0}\Phi -\tilde m_\ve\po_\xi \Phi.
\end{multline} 
Concerning equation \eqref{e5.11''} we observe that, with the exception of  
 \begin{equation}\label{e5.11''a}
 \po_{\un{y}_0} \left((\tilde a^0(\hat y,\xi)- \tilde a_\ve^0(\un{\hat y},\xi)) \tilde f_\ve^\Phi \right),
 \end{equation}
 all other terms in the right-hand side of this equation may be cast  in the form
 $\po_\xi \mu$ for some measure $\mu\in \M((0,\infty)\X\R^d\X(-L,L))$. Indeed, the first term may written as
 \begin{equation}\label{e5.11''b}
 -\ve\Phi \po_{\xi}\nabla_{x''}\left(\sgn(\tilde u_\ve-\xi)_+(f''(\tilde u_\ve)-f''(\xi))\right),
 \end{equation}
 where $\tilde u_\ve(\ul y)=\int_{\R}\tilde f_\ve(\ul y,\xi)\,d\xi$. Due to the assumption  \eqref{e1.a''}  and the hypothesis  that $\nabla_{x''}b(u)\in L^2(U_T)$, it follows that \eqref{e5.11''b} is $\Phi$ times   the $\xi$-derivative of a measure scaled in the same way as $m_\ve$ and so converging to zero by Claim~\#3.    That is, it has the form $\po_\xi(\Phi \mu_\ve)-\mu_\ve\po_\xi\Phi$, and so it can be rendered as $\po_\xi \tilde \mu_\ve$, for a measure $\tilde \mu_\ve$ converging to zero as $\ve\to0$.  
 The first, second and third terms in the third line of \eqref{e5.11''} can clearly be put in the form $\po_\xi\mu_\ve$ for a measure $\mu_\ve$ converging to
 $$
 \int_0^\xi \Bigl\{\tilde f_\infty \tilde a^0(\hat y,\xi)  \po_{\ul{y}_0}\Phi+ \tilde f_\infty \hat a'(\xi)\cdot\nabla_{\un{\hat y}'}\Phi-\tilde f_\infty B'(\xi)\cdot\nabla_{\un{\hat y}''}^2\Phi \Bigr\}\,d\z,
 $$
 as $\ve\to0$. 
  As for the last term in the third line of \eqref{e5.11''},  we first observe that, formally, we have
  \begin{multline*} 
   \s(\xi)\nabla_{\ul{y}''}\tilde f_\ve=\s(\xi)\nabla_{\ul{y}''}\tilde u_\ve \d_{\tilde u_\ve(\ul y)}= \nabla_{\ul{y}''}\int_0^{\tilde u_\ve(\ul{y}'')}\s(\z)\,d\z\,\d_{\tilde u_\ve(\ul{y}'')}(\xi) \\
   =-\ve^{1/2}\nabla_{x''}\int_0^{\tilde u_\ve(\ul{y}'')}\s(\z)\,d\z\,\po_{\xi} 1_{(0, \tilde u_\ve(\ul{y}''))}(\xi)\\
   =-\po_{\xi}\left(\ve^{1/2}\nabla_{x''}\int_0^{\tilde u_\ve(\ul{y}'')}\s(\z)\,d\z\, 1_{(0, \tilde u_\ve(\ul{y}''))}(\xi) \right).
 \end{multline*}
 Now the primitive $\Sigma(u)=\int_0^u\s(\z)\,d\z$ is a Lipschitz function of the function $b(u)$, that is, $\Sigma (u)=\tilde \Sigma(b(u))$ for some Lipschitz $\tilde \Sigma$ as follows from \eqref{e1.B'}.
 The formal calculation may be easily made rigorous. Using these facts  and other trivial rearrangements we may at last also render the third term in the third line of \eqref{e5.11''} in the form $\po_\xi\mu_\ve$ 
 for some measure $\mu_\ve$ converging to zero as $\ve\to0$. Finally, it is immediate that both terms in the last line of \eqref{e5.11''} can also be put in the form $\po_\xi \mu_\ve$ for some measure $\mu_\ve\wto0$ as $\ve\to0$. 
 
 As for the term in \eqref{e5.11''a}, we observe that it vanishes if $\Om'$ is a star like polyhedric domain (see Remark~\ref{R:2.3}),  and $\hat y$ corresponds  to a point in the smooth (flat) part of the boundary $\po\Om'$, since in this case we may choose $\g_\a=\text{const.}$ in \eqref{e.orient}.    Therefore, the claim follows in this case from Tadmor and Tao's averaging lemma~2.3 in \cite{TT}. 
 
 For the case where $\Om'$ is a general bounded open set in $\R^{d'}$ with smooth boundary, we apply the following lemma, observing that, if the reduced symbol $\LL_0(\tau, \k,\xi)$ satisfies \eqref{e1.nondeg}, then the reduced symbol
 $$
\tilde \LL_0(\tau,\k, \xi):= \tilde a^0(\hat y ,\xi)\tau +\hat a'(\xi)\cdot\k' -B'(\xi):(\k''\otimes\k'')
$$
 satisfies: for all $(\tau,\k',\k'')\in\R^{d+1}$, with $\tau^2+|\k'|^2+|\k''|^2=1$, we have
 \begin{equation}\label{e5.nondeg0}
 \operatorname{meas}\left\{\xi\in[-L_0,L_0]\,:\, |\tilde a^0(\hat y,\xi)\tau+ \hat a'(\xi)\cdot\k'|^2+\left(B'(\xi):(\k''\otimes\k'')\right)^2=0\right\}=0,
 \end{equation}
 as it is easy to verify.  The following result is from \cite{FLMNZ3}. It also follows from a general theory on averaging lemmas recently developed  by J.F.C.~Nariyoshi in his PhD thesis in \cite{Na}.

\begin{lemma}\label{L:5.1} Let $N,N',N''$ be positive integers with $N=1+ N'+N''$, $f_n(y,\xi)$ be a  bounded sequence in $L^\infty(\R^N\X\R)$,  such that $f_n(y,\xi)=0$, if $(y,\xi)\notin K\X[-L,L]$, where $K\subset \R^{N}$ is compact.  
 Let  $h_n$  be compact in $L^p(\R^N\X\R)$, $1<p<2$.   For $y\in\R^N$ we write $y=(y_0,y',y'')$, $y_0\in\R$, $y'\in\R^{N'}$, $y''\in\R^{N''}$. 
 Let  $\CS^{N''\X N''}$ denote the space of the $N''\X N''$ symmetric matrices, and let $(\a_0,\a')\in C^2([-L_0,L_0];\R^{1+N'})$, $\b\in C^2([-L_0,L_0];\CS^{N''\X N''})$, for some $L_0>L$.  Assume
  \begin{equation}\label{e5.14}
 \a_0(\xi)\po_{y_0}f_n+   \a'(\xi)\cdot\nabla_{y'}f_n-\b(\xi): \nabla_{y''}^2f_n= (1-\po_\xi^2)(1-\Delta_y)^{1/2} h_n, %  \po_\xi\nabla_{y,\xi}\cdot\bfg_n^1+\nabla_{y,\xi}\cdot\bfg_n^2,
 \end{equation}
where  $\a_0(\xi)\ne0$ and $\b(\xi)>0$, for a.e.\ $\xi\in [-L,L]$, and  the symbol $\LL(\tau,\k',\k'',\xi):=i( \a_0(\xi) \tau+\a'(\xi)\cdot \k')+\b (\xi)(\k'',\k'')$ satisfies: for all $\k:=(\tau,\k',\k'')\in\R^N$, with $\tau^2+|\k'|^2+|\k''|^2=1$, we have
\begin{equation}\label{e5.nondeg}
\operatorname{meas}\{\xi\in[-L,L]\,:\, |\a_0(\xi)\tau+\a'(\xi)\cdot\k'|^2+(\k'' \b(\xi)\k'')^2=0\}=0.
\end{equation} 
%Assume further, that for any $\varphi\in C_c^\infty(\R^N)$, 
%\begin{multline*}
%s(\xi)\nabla_{y''}\varphi \cdot \s(\xi) \nabla_{y''}f_n=\po_\xi m_\varphi,  \quad \text{ where $m_\varphi$ is bounded in $\M(\R^N\X\R)$}, \\ \text{and $\s(\xi)$ is such that $\b(\xi)=\s(\xi)^2$}.
%\end{multline*}   
Then, the average $u_n(y)=\int_{\R} f_n(y,\xi)\,d\xi$ is relatively compact in $L_\loc^1(\R^N)$. 
\end{lemma}

 \begin{proof}  Since $f_n$ is uniformly bounded with compact support, we may, with no loss of generality, assume that $f_n\wto 0$, weakly-star in $L^\infty(\R^N\X\R)$. Similarly, since $h_n$ is compact in $L^p(\R^N\X\R)$, $1<p<2$, we may assume that $h_n\to0$ in $L^p(\R^N\X\R)$.     
 Let $\z\in C_c^\infty(\C)$ be radially symmetric and such that $\z(z)=1$ for $|z|<1$ and $\z(z)=0$, for $|z|>2$, and $\psi(z)=1-\z(z)$, $z\in\C$. Let $\k\in \R^N$, $\k=(\tau, \k',\k'')$, $\tau\in\R$, $\k'\in\R^{N'}$, $\k''\in\R^{N''}$. Let us denote  $\tilde \a'(\xi)=(\a_0(\xi),\a'(\xi))$
 and $\tilde \k'=(\tau,\k')$. 
 For $\d>0$ and $\g>0$, let us denote
 \begin{equation*}
 \begin{aligned}
 &\z^{(1)}(\k):=\z\left(\frac{|\k|}{\g}\right),\\
 &\z^{(2)}( \k,\xi):=\z \left(\frac{i\tilde \a'(\xi)\cdot \tilde\k'}{\d|\tilde \k'|}\right),\\
 &\z^{(3)}(\k,\xi):=\z\left(\frac{\b(\xi)(\k'',\k'')}{\d |\k''|^2}\right),\\
 &\psi^{(1)}(\k):=1-\z^{(1)}(\k),\\
 &\psi^{(i)}(\k,\xi)=1-\z^{(i)}(\k,\xi),\qquad i=2,3.
 \end{aligned}
 \end{equation*}
 Let  $\fF$ denote the Fourier transform in $\R^N$. We have 
 \begin{equation} \label{e5.16'}
 \begin{aligned}
 \fF f_n&=\z^{(1)} \fF f_n+\psi^{(1)}\fF f_n\\
            &=\z^{(1)} \fF f_n+\psi^{(1)} \z^{(2)}\fF f_n+\psi^{(1)}\psi^{(2)} \fF f_n\\
            &=\z^{(1)}\fF f_n+  \psi^{(1)} \z^{(2)}\fF f_n+\psi^{(1)}\psi^{(2)} \z^{(3)} \fF f_n+ \psi^{(1)}\psi^{(2)} \psi^{(3)} \fF f_n\\
            &=: \fF f_n^{(1)} +\fF f_n^{(2)}+ \fF f_n^{(3)} +\fF f_n^{(4)}. 
    \end{aligned}
\end{equation}
Observe that $\fF f_n(\k,\xi)=0$, is $\xi\notin [-L,L]$, for all $\k\in\R^N$.  Let us also denote
\begin{equation} \label{e5.16''}
\begin{aligned}
&\int_{-L}^L  f_n^{(1)}\,d\xi=: v^{(1)},\\
&\int_{-L}^L  f_n^{(2)}\,d\xi=: v^{(2)},\\
&\int_{-L}^L  f_n^{(3)}\,d\xi=: v^{(3)},\\
&\int_{-L}^L  f_n^{(4)}\,d\xi=: v^{(4)}.
\end{aligned}
\end{equation}
Thus, 
\begin{equation}\label{e5.16'''}
\int_{-L}^L f_n\, d\xi= v^{(1)}+  v^{(2)}  +  v^{(3)} +  v^{(4)}.
\end{equation}

 Concerning $v^{(1)}$, we first observe that, since $ \sup_{\xi\in\R}  \|f_n(\cdot,\xi)\|_{L^1(\R^N)}\le C$, for some constant $C>0$ independent of $n$. Therefore, it follows that $ \sup_{\xi\in\R}\|\fF f_n(\cdot,\xi)\|_{L^\infty(\R^N)}\le C$. Thus, by Cauchy-Schwarz inequality,   we have
 \begin{equation*}
 \begin{aligned}
 \int_{\R^N} |v_n^{(1)} (x)|^2\,dx &\le 2L\int_{-L}^L\int_{\R^N} |\z^{(1)}(\k)|^2 |\fF f_n(\k,\xi)|^2\,d\k\,d\xi\\
                                           &\le 4L^2  C \int_{\R^N}  |\z^{(1)}(\k)|^2   \,d\k \\
                                           &\le  \tilde C \g^{N},\\
                                           \end{aligned}
  \end{equation*}
 for some constant $\tilde C>0$ independent of $n$ and $\g$. 
 
  Concerning $v^{(2)}$, we have
  \begin{equation*}
  \begin{aligned}
  \fF v_n^{(2)} (\k) &= \int_{\R}  \fF f_n^{(2)}(\k,\xi)\,d\xi\\
  &=\int_{\R}  \psi^{(1)}(\k)\z^{(2)}(\k,\xi) \fF f_n (\k,\xi)\,d\xi.
  \end{aligned}
  \end{equation*}
  By Plancherel identity and Cauchy-Schwarz 
  inequality, we have 
 \begin{equation*}
 \begin{aligned}
 \|v_n^{(2)}\|_{L^2(\R^N)}^2 &=\|\fF v_n^{(2)}\|_{L^2(\R^N)}^2\\
 &\le \int_{\R^N}\left(\int_{-L}^L  1_{\{|\tilde\a'(v)\cdot \tilde\k'|\le 2\d|\tilde\k'| \}}(\xi)\,d\xi\right)\\
 &\qquad\qquad \left(\int_\R \left|\psi^{(1)}(\k) \z^{(2)}(\k,\xi)\fF f_n(\k,\xi)\right|^2\,d\xi\right)\, d\k\\
 &\le \sup_{|\tilde\k'|=1}\left|\{\xi\in[-L,L] \,:\,|\tilde\a'(\xi)\cdot\tilde\k'| \le  2\d\}\right| \|f_n\|_{L^2(\R^N\X\R)}^2,
 \end{aligned}
 \end{equation*}
 where we denote by $|\{\cdots\}|$ the Lebesgue measure of $\{\cdots\}$.  Now, define the functions $h_\d: \mathbb{S}^{d'}:=
 \{|\tilde\k'|=1\}\to \R$  by
 $$
 h_\d(\tilde\k')= \left|\{\xi\in [-L,L]  \,:\,|\tilde\a'(\xi)\cdot\tilde\k'| \le  2\d\}\right|.
 $$ 
It is easy to check that $h_\d$ is continuous on $\mathbb{S}^{d'}$ and that $h_{\d_1}(\tilde\k')\le h_{\d_2}(\tilde\k')$, if $\d_1\le \d_2$, for all $\tilde\k'\in\mathbb{S}^{d'}$. Therefore, because of
the nondegeneracy condition \eqref{e5.nondeg} we deduce that $\sup_{|\tilde\k'|=1}  h_d(\tilde\k')\to 0$, as $\d\to0$. Therefore, we may write
\begin{equation}\label{e5.16iv}
 \|v_n^{(2)}\|_{L^2(\R^N)}^2 \le O(\d),
 \end{equation}
 where $O(\d)\to0$ as $\d\to0$, uniformly with respect to $n$.

 Similarly, for $v^{(3)}$, we have 
  \begin{equation*}
 \begin{aligned}
 \|v_n^{(3)}\|_{L^2(\R^N)}^2 &=\|\fF v_n^{(3)}\|_{L^2(\R^N)}^2\\
 &\le \int_{\R^N}\left(\int_{-L}^L 1_{\{\b(v)(\k'',\k'')\le 2\d|\k''|^2 \}}(\xi) \,d\xi\right)\\
 &\qquad\qquad \left(\int_\R \left|\psi^{(1)}(\k)\psi^{(2)}(\k,\xi) \z^{(3)}(\k,\xi)\fF f_n(\k,\xi)\right|^2\,d\xi\right)\, d\k\\
  &\le \sup_{|\k''|=1}\left|\{\xi\in[-L,L] \,:\,\b(\xi)(\k'',\k'') \le  2\d\}\right| \|\phi\|_{L^\infty(\R)}^2 \|f_n\|_{L^2(\R^N\X\R)}^2,
 \end{aligned}
 \end{equation*}  
  where we  have used again Plancherel identity and Cauchy-Schwarz inequality, and again by a reasoning similar to that used for $v_n^{(2)}$ we arrive at 
  \begin{equation}\label{e5.16v}
   \|v_n^{(3)}\|_{L^2(\R^N)}^2 \le O(\d),
 \end{equation}
 where $O(\d)\to0$ as $\d\to0$,  uniformly with respect to $n$.
 
 Let us now consider $v_n^{(4)}$.  Let $\z\in C_c^\infty(\R)$ be such that $\z(\xi)=1$, for $\xi\in[-L,L]$, and $\z(\xi)=0$ for $\xi>L_0$, for some $L_0>L$.  We then have
 \begin{equation}\label{e5.16vi}
 v_n^{(4)}= \int_{\R} \z(\xi) \fF^{-1}\Big( \psi^{(1)}(\k)\psi^{(2)}(\k,\xi)\psi^{(3)}(\k,\xi)\fF f_n\Big) (y,\xi) \, d\xi
 \end{equation}
 Since $\psi^{(2)}\psi^{(3)}$ vanishes on the null set  of the symbol 
 $$
 \LL(\k,\xi):= i\tilde\a'(\xi)\cdot \tilde\k' +\b(\xi)(\k'',\k'')
 $$ 
 we may use the equation \eqref{e5.14} to write
\begin{equation*}
  \psi^{(1)}(\k)\psi^{(2)}(\k,\xi)\psi^{(3)}(\k,\xi)\fF f_n= \frac{\psi^{(1)}(\k)\psi^{(2)}(\k,\xi)\psi^{(3)}(\k,\xi) (1+|\k|^2)^{1/2} }{\LL(\k,\xi)} (1-\po_\xi^2) \fF h_n.
  \end{equation*}  
      
 We now prove that $v_n^{(4)}$ is relatively compact in $L_\loc^1$. We have
 \begin{equation}\label{e5.16vi'}
v_n^{(4)}=  \int_{\R} \z(\xi) \fF^{-1}\Big( \frac{\psi^{(1)}(\k)\psi^{(2)}(\k,\xi)\psi^{(3)}(\k,\xi) (1+|\k|^2)^{1/2} }{\LL(\k,\xi)} (1-\po_\xi^2) \fF h_n\Big) (y,\xi) \, d\xi.
\end{equation} 
Performing an integration by parts in \eqref{e5.16vi'} we obtain
\begin{equation} \label{e5.16vii}
\begin{aligned}
 & v_n^{(4)}= \\
&\int_{\R} (\z(\xi)-\z''(\xi)) \fF^{-1}\Big( \frac{\psi^{(1)}(\k)\psi^{(2)}(\k,\xi)\psi^{(3)}(\k,\xi) (1+|\k|^2)^{1/2}}{\LL(\k,\xi)}  \fF h_n \Big) \,d\xi \\
&+\int_{\R}\z'(\xi) \fF^{-1} \Big(  \po_\xi \left[ \frac{\psi^{(1)}(\k)\psi^{(2)}(\k,\xi)\psi^{(3)}(\k,\xi) (1+|\k|^2)^{1/2}}{\LL(\k,\xi)}\right] \fF h_n \Big) \,d\xi\\
&-\int_{\R}\z(\xi) \fF^{-1} \Big(  \po_\xi^2 \left[ \frac{\psi^{(1)}(\k)\psi^{(2)}(\k,\xi)\psi^{(3)}(\k,\xi) (1+|\k|^2)^{1/2}}{\LL(\k,\xi)}\right] \fF h_n \Big) \,d\xi
\end{aligned}
\end{equation}
So, let us define
$$
\begin{aligned}
m_1(\k,\xi)&:=  \frac{\psi^{(1)}(\k)\psi^{(2)}(\k,\xi)\psi^{(3)}(\k,\xi) (1+|\k|^2)^{1/2}}{\LL(\k,\xi)} \\
m_2(\k,\xi)&:= \po_\xi \left[ \frac{\psi^{(1)}(\k)\psi^{(2)}(\k,\xi)\psi^{(3)}(\k,\xi) (1+|\k|^2)^{1/2}}{\LL(\k,\xi)}\right] \\
m_3(\k,\xi)&:= \po_\xi^2 \left[ \frac{\psi^{(1)}(\k)\psi^{(2)}(\k,\xi)\psi^{(3)}(\k,\xi) (1+|\k|^2)^{1/2}}{\LL(\k,\xi)}\right] 
\end{aligned}
$$ 
We are going to show that $m_1(\k,\xi)$, $m_2(\k,\xi)$  and $m_2(\k,\xi)$ are $L^p$ multipliers in $\R^N$, uniformly in $\xi\in[-L_0,L_0]$. 
For that, we are going to apply the multidimensional extension of Marcinkiewicz multiplier theorem as stated in \cite{Duo}, chapter 8:
Let $m$ be differentiable in all quadrants of $\R^n$ and satisfy
$$
\sup_{i_1,\cdots,i_k}\int_{I_{i_1}\X\cdots\X I_{i_k}}\left|\frac{\po^k m}{\po\k_{i_1}\cdots\po\k_{i_k}}(\k)\right|\,d\k_{i_1}\cdots\,d\k_{i_k}<\infty,
$$
where the $I_j$'s are dyadic intervals in $\R$ and the set $\{i_1,\cdots,i_k\}$ runs over all the subsets $\{1,\cdots, n\}$ containing $k$ elements, $1\le k\le n$.  Clearly, it suffices to show that
$$
\sup_{\k\in\R^n} \k_{i_1}\cdots\k_{i_k} \left|\frac{\po^k m}{\po\k_{i_1}\cdots\po\k_{i_k}}(\k)\right| <\infty,
$$
for all such $\{i_1,\cdots,i_k\}$. 

First, we observe that 
$$
|m_1(\k,\xi)|\le C,
$$
with $C>0$ independent of $(\k,\xi)$. This follows from the fact that in the region where $|\k|>\g$,  $|\tilde\a'(\xi)\cdot \tilde\k'|> \d|\tilde\k'|$ and $\b (\xi)(\k'',\k'')>\d|\k''|^2$, it is not difficult to check that
$|\LL(\k,\xi)|\ge C|\k|$, for some constant $C>0$. So the boundedness for $m_1$ follows from the boundedness of $\psi^{(1)}\psi^{(2)}\psi^{(3)}$. Similarly, using the same reasoning, 
and the fact that the $\xi$-derivatives of $\psi^{(2)}(\k,\xi)$ and $\psi^{(3)}(\k,\xi)$ are uniformly bounded in $(\k,\xi)$, we also deduce that   
$$
|m_2(\k,\xi)|\le C,\quad |m_3(\k,\xi)|\le C,
$$
with $C>0$ independent of $(\k,\xi)$.  
 
 Now, let us analyze $\k_{i} \po_{\k_i}m_1$, for $i\in \{0,1,\cdots,N'+N''\}$. We claim that these expressions are bounded. First, if the derivative hits $\psi^{(1)}$ then the boundedness is clear since $\po_{\k_i}\psi^{(1)}$ has support in $|\k|\le 2\d$. On the other hand,
 if $\k_i\in \{0,\cdots,N'\}$ and the derivative hits $\psi^{(2)}$, then it is easy to see also that the expression is bounded since the derivation of the argument of $\psi^{(2)}$ multiplied by $\k_i$ is bounded. Also, if the derivative $\po_{\k_i}$ hits $1/\LL(\k,\xi)$ then it is clear that 
 the derivative of $1/\LL(\k,\xi)$ multiplied by $\k_i$ is bounded, where  we use that the support of $m_1$ is in a region where 
 $|\k|>\g$,   $|\tilde\a'(\xi)\cdot \tilde\k'|> \d|\tilde\k'|$ and $\b (\xi)(\k'',\k'')>\d|\k''|^2$. Analogously, we see that if $i\in\{N'+1,\cdots, N'+N''\}$,
 $\k_{i} \po_{\k_i}m_1$ is bounded uniformly in $(\k,\xi)\in\R^N\X\R$.  Similarly,  we prove $\k_{i} \po_{\k_i}m_1$, $i=0,\cdots,N$,  is bounded uniformly in $(\k,\xi)\in\R^N\X\R$.  In this way, we may check the hypotheses of the extended Marcinkiewcz multiplier theorem for $m_i$, $i=1,2,3$,  and conclude that they are satisfied uniformly for $\xi\in[-L_0,L_0]$. Hence, we deduce that
 \begin{equation}\label{efinal}
 \|v_n^{(4)}\|_{L^p(\R^N)}\le C\|\z\|_{C^2([-L_0,L_0])} \|h_n\|_{L^p(\R^N\X\R)}.
 \end{equation}
 Therefore, we conclude the proof of the lemma as follows. Given $\ve>0$ we may choose $\gamma>0$ and $\d>0$ such that 
 $\|v_n^{(i)}\|_{L^2(\R^N)}<\ve$, $i=1,2,3$, uniformly in $n\in\N$.  Then, making $n\to\infty$ we get that $v_n^{(4)}$ converges to 0 in $L^p(\R^N)$. Then, since $\ve>0$  is arbitrary, we see that $v^{(i)}\to0$, in $L_\loc^1(\R^N)$, for $i=1,2,3$, which concludes the proof.              
 
 \end{proof}

Step \#5. In this step one provides a characterization of the limit function $\tilde f_\infty$. The next claim is the same as proposition~4 in \cite{Va} and its proof is also identical to that of proposition~4 in \cite{Va}.  We first observe that,
by  \eqref{e5.13},   $\tilde a^0(\hat y,\xi)\tilde f_\infty$ lies in $C^0(\R_+,W^{-2,\infty}(\R^d\X(-L,L)))$.  We also recall that  $\tilde a^0(\hat y,\xi)\ne0$, for a.e.\ $\xi\in\R$, due to \eqref{e1.nondeg},  as already observed.

{\em Claim \#6}. For every $\hat y\in\E'$ we have
\begin{equation}\label{e5.17}
\tilde f_\infty(0,\un{\hat y},\xi)=f^\tau(\hat y,\xi)
\end{equation}
for almost all $(\un{\hat y},\xi)\in\R^d\X(-L,L)$. In particular, $\tilde f_\infty(0,\cdot,\cdot)$ does not depend on $\un{\hat y}$. 

\bigskip

For a fixed $\xi$, so that $\tilde a_0(\hat y,\xi)\ne 0$ and $B'(\xi)>0$,  we make the following change of coordinates in \eqref{e5.13}:  $\un{z}_0=\tilde a_0(\hat y,\xi) \un{y}_0$, $\un{\hat z}'=\un{\hat y}'- \hat a'(\xi) \un{z}_0$, $\un{\hat z}''=\un{\hat y}''$. We then  see that $\tilde f_\infty$ satisfies 
\begin{equation}\label{e5.finfty}
\po_{\un{z}_0}\tilde f_\infty- B'(\xi):\nabla_{\un{\hat z}''}^2\tilde f_\infty=0.
\end{equation}
We observe that $\un{z}_0$ is positive or negative depending on whether $\tilde a_0(\hat y,\xi)$ is positive or negative. In any case, local smooth regularity  of any solution of \eqref{e5.finfty} is guaranteed. Since, by Claim~\#5, $f_\infty$ is a $\chi$-function,
we deduce that it cannot depend on ${\hat z}''$ which, together with \eqref{e5.17}, gives\footnote{The authors thank J.F.~Nariyoshi for the idea of this argument using regularity.} 
$$
\tilde f_\infty (\un{y},\xi)=f^\tau(\hat y,\xi),
$$
for almost all $(\un{y},\xi)\in\R^{d+1}\X(-L,L)$, which is constant with respect to $\un{y}$. Hence, since $\tilde f_\infty$ is a $\chi$-function for almost all $\un{y}$, we finally arrive at the following conclusion.

{\em Claim \#6}. For every $\hat y\in\E'$ the function $f^\tau(\hat y,\cdot)$ is a $\chi$-function.

The conclusion of the proof of Theorem~\ref{T:5.1} is now exactly as the conclusion of the proof of theorem~1 in \cite{Va}, the difference here being that we are dealing with deformations $\Psi:\Gamma_T'\X[0,1]\to\bar U_T$ satisfying \eqref{e5.Psi}, where $\Psi'$ is a strongly regular deformation of $\po\Om'$.

\end{proof}

\section{Existence of solution to \eqref{e1.1}--\eqref{e1.4}} \label{S:4}

In this section we prove the existence of entropy solutions to the problem \eqref{e1.1}--\eqref{e1.4}, according to Definition~\ref{D:1.1}. 

\begin{theorem}\label{T:4.1} There exists an entropy solution to the problem \eqref{e1.1}--\eqref{e1.4}. 
\end{theorem}

\begin{proof} As usual, we approximate the solution of \eqref{e1.1}--\eqref{e1.4} by solutions of the following corresponding problem for a non-degenerate parabolic equation
\begin{align}
&u_t+\nabla\cdot \bff(u)= \nabla_{x''}\cdot (B'(u)\nabla_{x''}u)+\ve\Delta u, \quad (t,x)\in (0,\infty)\X\Om,\label{e4.1} \\
&u(0,x)=u_0(x), \quad x\in\Om,\label{e4.2}\\
&u(t,x)=a_0(x), \quad x\in \Om'\X\po\Om'', \ t>0, \label{e4.3}\\
& (\bff(u(t,x))-\ve\nabla u)\cdot \nu(x)= 0,\quad x\in \po\Om'\X\Om'',\ t>0. \label{e4.4}
\end{align}
Problems of the type of \eqref{e4.1}--\eqref{e4.4} are solved in detail in \cite{LSU}.  We denote by $u^\ve(t,x)$ the solution of \eqref{e4.1}--\eqref{e4.4}.  We divide the rest of the proof into seven steps. 

\medskip
Step \#1. We claim that $u_{\min}\le u^\ve(t,x)\le u_{\max}$. 

Indeed, the solution $u^\ve$ of \eqref{e4.1}--\eqref{e4.4} can be obtained as uniform limit as $\mu\to0$ of solutions  of the problem  
\begin{align}
&u_t+\nabla\cdot \bff(u)= \nabla_{x''}\cdot( B'(u)\nabla_{x''} u)+\ve\Delta u+\mu h(u), \quad (t,x)\in (0,\infty)\X\Om,\label{e4.1mu} \\
&u(0,x)=u_0^\mu(x), \quad x\in\Om,\label{e4.2mu}\\
&u(t,x)=a_0^\mu (x), \quad x\in \Om'\X\po\Om'', \ t>0, \label{e4.3mu}\\
& (\bff(u(t,x))-\ve\nabla u)\cdot \nu(x)= \mu(u-u^b(t,x)),\quad x\in \po\Om'\X\Om'',\ t>0. \label{e4.4mu}
\end{align}
where $h(u)$ is a smooth function satisfying $h(u_{\min})>0$, $h(u_{\max})<0$, $u^b(t,x)$ is any continuous function on $\po\Om'\X\Om''$ assuming values in $[u_{\min}+\mu,
u_{\max}-\mu]$. Moreover, we assume that $a_0^\mu$ is a sequence of smooth functions over $\Om'\X\po\Om''$ assuming values in $[u_{\min}+\mu,u_{\max}-\mu]$ converging uniformly to $a_0$ over $\Om'\X\po\Om''$. Finally, we assume that $u_0^\mu$ are smooth functions converging in $L^1(\Om)$ to $u_0$ as $\mu\to0$ and assuming values in 
$[u_{\min}+\mu, u_{\max}-\mu]$.    It suffices then to prove that the solutions $u^{\ve,\mu}$  of \eqref{e4.1mu}--\eqref{e4.4mu} satisfy $u_{\min}\le u^{\ve,\mu}(t,x) \le u_{\max}$, for all $(t,x)\in U_T$. The latter is proven as follows.

If the assertion is not true, because of the hypotheses on $u_0^\mu$ and $a_0^\mu$, there is a  time $t_*\in(0,T)$ in which $u^{\ve,\mu}$ assumes one of the values $u_{\min}$ or $u_{\max}$ for the first time. Let $u^{\ve,\mu}(x_*,t_*)\in\{u_{\min},u_{\max}\}$. We easily see that $x_*\notin\Om$, since otherwise we get a contradiction using the equation \eqref{e4.1mu}, as usual. It remains the possibility that $x_*\in\po\Om'\X\Om''$. However, this would also lead us to a contradiction. Indeed, we have $\pi_{d'}(\bff)(u_{\min})=
\pi_{d'}(\bff)(u_{\max})=0$ and,  from \eqref{e4.4mu}, in the case $u^{\ve,\mu}(t_*,x_*)=u_{\min}$, for  $x_*\in\po\Om'\X\Om''$, we get $\nabla u^{\ve,\mu}(x_*,t_*)\cdot\nu(x_*)>0$, which  is a contradiction since $\nu(x_*)$ is  outward
 pointing and $u^{\ve,\mu}$ cannot assume values less than $u_{\min}$ by assumption. A similar reasoning shows that we get a contradiction if  $u^{\ve,\mu}(t_*,x_*)=u_{\max}$, for  $x_*\in\po\Om'\X\Om''$.

\medskip
 Step \#2. We claim that  $\nabla_{x''}b(u^\ve)$ is bounded in $L^2(U_T)$, uniformly in $\ve>0$.  Also, for any $B''$ as in Definition~\ref{D:1.1},   for all $\tilde \varphi \in C_0^\infty((0,T)\X \Om'\cap B'')$,  for all $\ve>0$, we have 
\begin{multline}\label{e4.H0}
\int_{\Om'}|b(u^\ve(t,x',x''))-b(a_0(x',x''))|\tilde\varphi(t,x',x'') \,dx'\\ \in \text{bounded subset of $L^2((0,T); H_0^1(\Om''))$}.
\end{multline}
Moreover,
\begin{equation}\label{e4.epsux}
\ve^{1/2} |\nabla u^\ve|\in \text{ bounded set in $L^2(U_T)$}.
\end{equation}

 For simplicity we consider an extension $\tilde a_0$ of $a_0$ to $\bar\Om$, such that $\tilde a_0\in C^2(\bar\Om)$.    We first multiply \eqref{e4.1} by $u^\ve-\tilde a_0$, and write
\begin{gather*}
(\frac12u^{\ve2})_t-(\tilde a_0u)_t+\nabla\cdot {\bf g}(u^\ve)-\nabla\cdot (\tilde a_0\bff(u^\ve))+\nabla \tilde a_0\cdot\bff(u^\ve)\\
= (u^\ve-\tilde a_0)\nabla_{x''}\cdot (B'(u^\ve)\nabla_{x''}u)+\ve (u^\ve-\tilde a_0)\Delta(u^\ve-\tilde a_0)\\
-(u^\ve-\tilde a_0)(\Delta_{x''}b(\tilde a_0)+\ve\Delta\tilde a_0),
\end{gather*}
where $d{\bf g}/du=u(d\bff/du)$. We then integrate the above equation over $U_T$, use integration by parts in the integral of the terms in the second line,  applying \eqref{e4.3} and \eqref{e4.4}, to get that all boundary terms are bounded, and so we obtain after routine manipulations
$$
\int_{U_T} \ve |\nabla u^\ve|^2\,dx\,dt\le C,
$$
and 
$$
\int_{U_T} B'(u^\ve)\nabla_{x''}u^\ve\nabla_{x''} u^\ve\le C,
$$
for some constant $C$ depending only on $\|u^\ve\|_\infty,\bff, b, \|\tilde a_0,\nabla\tilde a_0,\Delta \tilde a_0\|_\infty, \Om$. The first of these inequalities is \eqref{e4.epsux}. {}From the second, using \eqref{e1.b1}, it follows that $\nabla_{x''}b(u^\ve)$ is bounded in $L^2(U_T)$ uniformly in $\ve>0$, and, so, also \eqref{e4.H0} follows.

\medskip
Step \#3. The solution $u^\ve$ of \eqref{e4.1}--\eqref{e4.4} satisfies for all $0\le \varphi\in C_0^\infty(U_T)$, 
\begin{multline}\label{e4.5E}
\int_{U_T}\{\eta_\d(u^\ve-k)\po_t\varphi-K_{x''}^\d(u^\ve,k)\cdot\nabla \varphi -\ve\nabla \eta_\d(u^\ve-k)\cdot\nabla \varphi \}\,dx\,dt 
\\ \ge \int_{U_T}\sgn_\d'(u^\ve-k)\sum_{k=d'+1}^d\left(\sum_{i=d'+1}^d\po_{x_i}\b_{ik}(u^\ve)\right)^2\varphi\,dx\,dt,
\end{multline}
where $\sgn_\d(u)=\sgn(u)$, for $|u|>\d$, and $\sgn_d(u)= \cos(\pi u/\d)$, for $|u|\le \d$.  $\eta_\d'(u)=\sgn_\d(u)$,  $K_{x''}^\d(u,v):=  \nabla_{x''}\cdot \Bbf^\d(u,v) -F^\d(u,v)$, with $\po_u F^\d(u,v)=\sgn_\d(u-v)\bff'(u)$,
$\Bbf^\d(u,v)=(\int_v^u \sgn_\d(s-v)b_{ij}'(u)\, ds)_{i,j=1}^d$.

The proof of \eqref{e4.5E} is by now somewhat standard. Namely, we multiply \eqref{e4.1} by $\sgn_\d(u^\ve-k)$, use chain rule, multiply by a test function $0\le \varphi\in C_0^\infty(U_T)$, integrate by parts, use \eqref{e1.bij}, etc.    We also use the  relation   $\sgn_\d(u^\ve-k)\Delta u^\ve\le \Delta\eta_\d(u^\ve-k)$, which follows immediately from the convexity of $\eta_\d$ .  In this way we obtain \eqref{e4.5E}.

\medskip
Step \#4.  (Compactness) We claim that $u^\ve$ converges in $L_{\loc}^1(U_T)$ to some function $u(t,x)$ as $\ve\to0$, which satisfies \eqref{e1.8E-1}, \eqref{e1.8E}, \eqref{e1.8N}, \eqref{e1.7} and \eqref{e1.8D'}.

As in Step~\#1 of the proof of Theorem~\ref{T:5.1}, defining
$$
 f^\ve(t,x,\xi)=\chi(\xi; u^\ve(t,x)),\quad \text{where}\quad \chi(\xi;u):=\begin{cases} -1, \quad u\le \xi< 0,\\ 1,\quad  0<\xi \le u,\\ 0,\quad |\xi|>|u|,\end{cases}
 $$
we prove that $f^\ve(t,x,\xi)$  satisfies
 \begin{equation} \label{e4.6}
 \po_t f^\ve+\bfa(\xi)\cdot\nabla f^\ve-B'(\xi):\nabla_{x''}^2 f^\ve-\ve\Delta f^\ve=\po_{\xi} m^\ve,
 \end{equation}
 in the sense of distributions in $\DD'(U_T\X(-L,L))$, with $\bfa(\xi)=\bff'(\xi)$, for $m^\ve\in\M_+(U_T\X(-L,L))$ with  total variation uniformly bounded in compacts with respect to $\ve$. Also, using Step~\#2 and hypothesis \eqref{e1.a''}, we indeed have that 
 $f^\ve$ satisfies
 \begin{equation} \label{e4.6'}
 \po_t f^\ve+\bfa'(\xi)\cdot\nabla_{x'} f^\ve-B'(\xi):\nabla_{x''}^2 f^\ve-\ve\Delta f^\ve=\po_{\xi} \tilde m^\ve,
 \end{equation}
 with $\bfa'(\xi)=(a_1(\xi),\cdots,a_{d'}(\xi))$ and $\tilde m^\ve\in\M_+(U_T\X(-L,L))$ with  total variation uniformly bounded in compacts with respect to $\ve$. 
 We can then apply the Lemma~\ref{L:5.1}  to obtain the compactness of $u^\ve$ is $L_{\loc}^1(U_T)$.    By extracting a subsequence still denoted by $u^\ve$, we obtain $u\in L^\infty(U_T)$
 such that $u^\ve\to u$ in $L_{\loc}^1(U_T)$.  
 The fact that $u$ satisfies \eqref{e1.8N} follows by passing to the limit as $\ve\to0$ in  
  \begin{equation}\label{e4.7N}
\int_{U_T}\{ u^\ve\po_t\tilde\phi +\bff(u^\ve)\cdot\nabla \tilde\phi- \nabla_{x''}\cdot B(u^\ve)\cdot\nabla_{x''}\tilde\phi -\ve\nabla u^\ve\cdot\nabla \tilde\phi\}\,dx\,dt=0,
\end{equation}
for all $\tilde \phi\in C_0^\infty((0,T)\X\R^{d'}\X\Om'')$, which is trivially obtained from \eqref{e4.1} and \eqref{e4.4}, observing that $\nabla_{x''} b_{ij}(u^\ve)$ weakly converges in $L^2(U_T)$ to $\nabla_{x''}b_{ij}(u)$, since this is true in the sense of distributions and $\nabla_{x''}b_{ij}(u^\ve)$ is weakly compact in $L_{\loc}^2(U_T)$. 
To prove that $u$ satisfies \eqref{e1.8E-1}, we first prove that $u^\ve$ satisfies an inequality similar to \eqref{e4.5E} with $\eta(u^\ve)$ instead of $\eta_\d(u^\ve-k)$, which is proved exactly in  the same way. We then pass to the limit when $\ve\to 0$ and use the inequality
\begin{multline*}
\liminf_{\ve\to0} \int_{U_T}\eta''(u^\ve)\sum_{k=d'+1}^d\left(\sum_{i=d'+1}^d\po_{x_i}\b_{ik}(u^\ve)\right)^2\varphi\,dx\,dt\\
\ge \int_{U_T}\eta''(u)\sum_{k=d'+1}^d\left(\sum_{i=d'+1}^d\po_{x_i}\b_{ik}(u)\right)^2\varphi\,dx\,dt
\end{multline*}
which follows from the strong convergence of $u^\ve$ to $u$ in $L^2(U_T)$, the boundedness of $\sum_{i=d'+1}^d\po_{x_i}\b_{ik}(u^\ve)$ in $L^2(U_T)$, for all $k=d'+1,\cdots,d$, and the weak lower semicontinuity of the $L^2$-norm.  
The fact that $u(t,x)$ satisfies \eqref{e1.8E} follows by passing to the limit as $\ve\to0$ in \eqref{e4.5E} and then passing to the limit when $\d\to0$. In passing to the limit as $\ve\to0$, we proceed exactly as it has just been done for the proof of  \eqref{e1.8E-1}.   

 We also see that $u(t,x)$ satisfies \eqref{e1.7} and \eqref{e1.8D'}, as a direct consequence of the uniform boundedness of  $\nabla_{x''} b(u^\ve)$ in $L^2(U_T)$ and \eqref{e4.H0} proved in Step~\#2.  

\medskip
Step \#5. The limit function $u(t,x)$ satisfies \eqref{e1.9}, in Definition~\ref{D:1.1}. 

Indeed, we first write \eqref{e4.1} as
$$
(u^\ve-a_0)_t+\nabla\cdot (\bff(u^\ve)-\bff(a_0))= \nabla_{x''}^2: (B(u^\ve)-B(a_0))+\ve\Delta (u^\ve-a_0)+ G_\ve,
$$
where $G_\ve=\nabla_{x''}^2: B(a_0)+\ve\Delta a_0-\nabla\cdot \bff(a_0)$.   We multiply the above equation by $\sgn(u-a_0)$. It is precisely here that we need to use condition \eqref{e1.C}, for in this case we have
$$
\sgn(u^\ve-a_0)\nabla_{x''}^2: (B(u^\ve)-B(a_0))\le \nabla_{x''}^2:\sgn(u^\ve-a_0)(B(u^\ve)- B(a_0)),
$$
while  the fact that $\sgn(u-a_0)\Delta(u-a_0)\le \Delta|u-a_o|$ is well known, both in the sense of the distributions in $(0,\infty)\X\R^d$. Therefore, we get
$$
|u^\ve-a_0|_t+\nabla\cdot \sgn(u^\ve-a_0)(\bff(u^\ve)-\bff(a_0))\le \nabla_{x''}^2: \Bbf(u^\ve,a_0)+\ve\Delta |u-a_0|+ \theta_\ve,
$$
in the sense of distributions, where $\theta_\ve=\sgn(u^\ve-a_0)G_\ve$. We then apply this relation to  $\varphi\in C_0^\infty (V_T)$, $V_T=(0,T)\X\Om'\X(\Om''\cap B'')$,  to get 
\begin{multline}\label{e4.8D1}
\int_{U_T}\{|u^\ve(t,x)-a_0(x)|\po_t\varphi+ F(u^\ve(t,x),a_0(x))\cdot\nabla\varphi \\-\Bbf(u^\ve,a_0):\nabla_{x''}^2\varphi
 -\ve |u^\ve(t,x)-a_0(x)|\Delta \varphi\}\,dx\, dt\ge -\int_{U_T} \theta^\ve \varphi\,dx\,dt,
\end{multline} 
Now, by approximation, we may take $\varphi=\z_{\d}''(x'')\tilde\varphi$, where $\z_\d''(x'')$ is a $\Om''$-canonical local boundary layer sequence and $\tilde \varphi\in C_0^\infty((0,T)\X\Om'\X B'')$. We then get 
\begin{multline}\label{e4.8D1'}
\int_{U_T}\{|u^\ve-a_0|\z_\d''\po_t\tilde\varphi+ F(u^\ve,a_0)\cdot \z''_\d\nabla\tilde\varphi +\Bbf(u^\ve,a_0):\z''_\d\nabla_{x''}^2\tilde\varphi\\
 +\ve |u^\ve-a_0|\z_\d''\Delta \tilde\varphi+ F(u^\ve,a_0)\cdot\nabla_{x''}\z_\d'' \tilde\varphi\\
+2\Bbf(u^\ve,a_0)\nabla_{x''}\z''_\d \nabla_{x''}\tilde\varphi  +2\ve |u^\ve-a_0|\nabla_{x''}\z_\d''\cdot\nabla_{x''} \tilde\varphi\}\,dx\,dt \\
+\int_{U_T}\{\tilde\varphi  \Bbf(u^\ve,a_0): \nabla_{x''}^2\z''_\d  +\ve |u^\ve-a_0|\tilde \varphi \Delta_{x''}\z_\d'' \}\,dx\,dt \\ 
\ge -\int_{U_T} \theta^\ve\z''_\d\tilde \varphi\,dx\,dt.
\end{multline} 
%%%%%%%%%%%%%%%%%%%%%%%%%%%%%%%%%%%%%%%%%%%%%%%%%%%%%%%%%%%%%%%%%%%%%%%%%%%%%%%%%%%%%%
Now, observe that the second term in the integral in the fourth line in the expression above is non-positive.  As for the first term,  by \eqref{e1.R1'},  it is the sum of two terms, the first of which is 
$\tilde\varphi \l\sum_{i,j}\sgn(u-a_0)(b_{ij}(u)-b_{ij}(a_0)) \nu_i \nu_j \,d\H^{d-1}\,dt$, which is non-negative since $B'(u):\nu\otimes\nu\ge0$  by \eqref{e1.B'}. The other term, still according to \eqref{e1.R1'}, converges to zero as $\d\to0$, since $\Bbf(u^\ve,a_0)$ vanishes at the boundary region $(0,T)\X\Om'\X\po\Om''$.  
%%%%%%%%%%%%%%%%%%%%%%%%%%%%%%%%%%%%%%%%%%%%%%%%%%%%%%%%%%%%%%%%%%%%%%%%%%%%%%%%

We then make $\d\to0$,  use the fact that $F(u^\ve,a_0)$, $\Bbf(u^\ve,a_0)$ and $|u^\ve-a_0|$ vanish on $(0,T)\X\Om'\X\po\Om''$, to get
\begin{multline*}
\int_{U_T}\{|u^\ve(t,x)-a_0(x)|\po_t\tilde\varphi+ F(u^\ve(t,x),a_0(x))\cdot \nabla\tilde\varphi \\ +\Bbf(u^\ve,a_0): \nabla_{x''}^2\tilde\varphi
 +\ve |u^\ve(t,x)-a_0(x)|\Delta \tilde\varphi\}\,dx\,dt  
\ge -\int_{U_T} \theta^\ve \varphi\,dx\,dt.
\end{multline*} 
We then integrate by parts the third term in the inequality above and make $\ve\to0$, noticing that  $\theta^{\ve}$ is uniformly bounded in $L^\infty(V_T)$, to get \eqref{e1.9}.

\medskip
Step \#6. The limit function  $u(t,x)$ satisfies  \eqref{e1.9'}. 

Let us define $\Bbf^*(u,k,a_0)$ as in Lemma~\ref{L:1.2} and $\F(u,k,a_0)=F(u,k)+F(u,a_0)-F(a_0,k)$. Arguing as above, from \eqref{e4.1} we obtain
\begin{equation*}
A(u^\ve,k,a_0)_t+ \nabla\cdot \F(u^\ve,k,a_0)\le \nabla_{x''}^2: \Bbf^*(u^\ve,k,a_0)+\ve\Delta A(u^\ve,k,a_0)+ \mu_0,
\end{equation*}
 in the sense of the distributions, where $\mu_0=|\div K_{x''}(a_0,k)|$. We proceed as in Step~\#5, first applying this inequality to $\varphi\in C_0^\infty (V_T)$, obtaining
 \begin{multline*}
 \int_{U_T} \{A(u^\ve,k,a_0)\varphi_t+\F(u^\ve,k,a_0)\cdot\nabla\varphi+\Bbf^*(u^\ve,k,a_0):\nabla_{x''}^2\varphi +\ve A(u^\ve,k,a_0)\Delta\varphi\}\,dx\,dt\\
 \ge -\int_{U_T}\varphi \,d\mu_0\,dt.
 \end{multline*}
 Then, by approximation, as in Step~\#5, we take $\varphi=\z_{\d}''(x'')\tilde\varphi$, where $\z_\d''(x'')$ is a $\Om''$-canonical local boundary layer sequence and $\tilde \varphi\in C_0^\infty((0,T)\X\Om'\X B'')$, to obtain 
\begin{multline}\label{e4.8D1''}
\int_{U_T}\{A(u^\ve,k,a_0)\z_\d''\po_t\tilde\varphi+ \F(u^\ve,k,a_0)\cdot \nabla\tilde\varphi \z''_\d+\Bbf^*(u^\ve,k,a_0):\nabla_{x''}^2\tilde\varphi \z''_\d\\
 +\ve A(u^\ve,k,a_0)\z_\d''\Delta \tilde\varphi+ \F(u^\ve,k,a_0)\cdot\nabla_{x''}\z_\d'' \tilde\varphi\\
+2(\Bbf^*(u^\ve,k,a_0)\nabla_{x''}\z''_\d )\cdot \nabla_{x''}\tilde\varphi  +2\ve A(u^\ve,k,a_0)\nabla_{x''}\z_\d''\cdot\nabla_{x''} \tilde\varphi\}\,dx\,dt \\
+\int_{U_T}\{ \tilde\varphi \Bbf^*(u^\ve,k,a_0): \nabla_{x''}^2\z''_\d  +\ve A(u^\ve,ka_0)\tilde \varphi \Delta_{x''}\z_\d'' \}\,dx'\,dt \\ 
\ge -\int_{U_T} \theta^\ve \z_\d''\tilde \varphi\,dx\,dt-\int_{U_T}\z''_\d\tilde\varphi \,d\mu_0\,dt.
\end{multline} 
Again we can discard the terms containing $\nabla_{x''}^2\z''_\d$ and $\Delta_{x''}\z''_\d$ in the fourth line of \eqref{e4.8D1''}, using \eqref{e1.R1'},  since 
$A(u^\ve,k,a_0)\ge0$, both $A(u^\ve,k, a_o)$ and $\Bbf^*(u^\ve,k,a_0)$ vanish at the boundary region $(0,T)\X\Om'\X\po\Om''$, and, by Lemma~\ref{L:1.2},
$$
\Bbf^*(u^\ve,k,a_0):\nu\otimes\nu\ge0. 
$$
Then, we make $\d\to0$, recalling that $\F(u^\ve,k,a_0)$, $A(u_0^\ve,k,a_0)$ and $\Bbf^*(u^\ve,k, a_0)$ vanish on  $(0,T)\X\Om'\X\po\Om''$, to get
\begin{multline*}
\int_{U_T}\{A(u^\ve(t,x),k,a_0(x))\po_t\tilde\varphi+ \F(u^\ve(t,x),k,a_0(x))\cdot \nabla\tilde\varphi +\Bbf^*(u^\ve,k,a_0):\nabla_{x''}^2\tilde\varphi
\\ +\ve A(u^\ve(t,x),k,a_0(x))\Delta \tilde\varphi  
 \ge -\int_{U_T} \theta^\ve \tilde \varphi\,dx\,dt- \int_{U_T}\tilde\varphi \,d\mu_0\,dt.
\end{multline*} 
Finally, we use integration by parts in the third term and send $\ve\to0$, recalling that $\theta^\ve$ is uniformly bounded in $L^\infty(V_T)$, to obtain \eqref{e1.9'}, as desired. 

\medskip
Step \#7. Finally, we claim that the limit function $u(t,x)$ satisfies the initial condition  \eqref{e1.10}.

Indeed, first we observe that, in the same way that we obtained that $u(t,x)$ satisfies \eqref{e1.8E}, we also obtain that $u(t,x)$ satisfies, for all  $0\le \varphi\in C_0^\infty((-\infty,T)\X\Om)$ and $k\in\R$,
\begin{equation}\label{e4.12}
\int_{U_T}\{|u-k|\po_t\varphi-K_{x''}(u,k)\cdot\nabla \varphi\}\,dx\,dt +\int_{\Om}|u_0(x)-k|\varphi(0,x)\,dx \ge0 ,
\end{equation}
from which, in particular, we deduce that 
\begin{equation}\label{e4.12'}
\int_{U_T}\{u\po_t\varphi-(\nabla_{x''}\cdot B(u)-\bff(u))\cdot\nabla \varphi\}\,dx\,dt +\int_{\Om} u\varphi(0,x)\,dx =0 ,
\end{equation}
holds for all $\varphi\in C_0^\infty((-\infty,T)\X\Om)$ and $k\in\R$.
Choosing, in \eqref{e4.12} and \eqref{e4.12'},  $\varphi(t,x)=\phi(x)(1-\z_\d(t))$, with $\phi\in C_0^\infty(\Om)$,  where $\z_\d(t)=\z(\d^{-1}t)$ and $\z(t)$ is a smooth function satisfying $0\le \z(t)\le 1$, $\z(0)=0$, $\z(t)=1$, for $t>1$, from \eqref{e4.12} and \eqref{e4.12'}, making $\d\to0$, using the normal trace formula in \eqref{e11},  since $G^k(t,x):=(|u(t,x)-k|,-K_{x''}(u(t,x),k))\in\DM^2(U_T)$, for all $k\in\R$,  and, in particular, $G(t,x):=(u(t,x), -\nabla_{x''}\cdot B(u)+\bff(u))\in \DM^2(U_T)$,  we deduce the following, concerning the normal traces restricted to the $\{t=0\}\X\Om$,   
\begin{gather*}
G^k\cdot\nu|\{t=0\}\X\Om\le |u_0(x)-k|,\\
 G\cdot\nu |\{t=0\}\X\Om=u_0(x),
\end{gather*}
which extends from a relation for functionals in $\Lip_0(\{t=0\}\X\Om)^*$ to measures and then as functions in $L^\infty(\Om)$. Now, by the trace formula \eqref{e2.3''}, we deduce that 
\begin{gather*}
 G^k\cdot\nu|\{t=0\}\X\Om=\esslim_{t\to0+} |u(x,t)-k|,\\
 G\cdot\nu |\{t=0\}\X\Om=\esslim_{t\to0+} u(t,x),
\end{gather*}
first in the sense of $\Lip_0(\{t=0\}\X\Om)^*$, which then extends to the weak star topology of $L^\infty((\{t=0\}\X\Om)$. In particular, we have
\begin{gather*}
\esslim_{t\to0+} |u(x,t)-k|\le |u_0(x)-k|,\\
\esslim_{t\to0+} u(t,x)=u_0(x),
\end{gather*}
in the sense of the weak star topology of $L^\infty(\{t=0\}\X\Om)$, which, by convexity, gives
$$
\esslim_{t\to0+} |u(t,x)-k|= |u_0(x)-k|.
$$
These facts together, imply that $u(t,x)$ converges a.e.\  as $t\to0+$ to $u_0$, which implies \eqref{e1.10},  by dominated convergence. This concludes the proof. 

\end{proof}

\section{Uniqueness of solution to \eqref{e1.1}--\eqref{e1.4}}\label{S:5}

In this section we prove the uniqueness part of Theorem~\ref{T:1.1}. The proof follows closely the one in \cite{MPT}, in which concerns the part of the boundary submitted to the Dirichlet boundary condition. As we will see, the part of the boundary submitted to the Neumann condition is handled more easily because of the strong trace property. 
We will use the following lemma, which is an easy extension of a formula proved in \cite{Ca},  stated as a lemma in \cite{MPT} (see also \cite{BK}). For the rest of this section, we agree that $Q:=U_T$.  

\begin{lemma}[cf.\ \cite{Ca}, \cite{MPT}, \cite{BK}]\label{L:6.1} Let $\xi(x,t,y,s)$ be a nonnegative $C^\infty$ function such that:
\begin{align*}
&(t, x)\mapsto \xi(t,x,s,y)\in C_c^\infty(Q) \quad \text{for every $(s,y)\in Q$},\\
&(s,y)\mapsto\xi(t,x,s,y)\in C_c^\infty(Q)\quad \text{for every $(t,x)\in Q$}.
\end{align*}
Let $u(t,x)$ and $v(s,y)$ satisfy  \textrm{(ii)} of Definition~\ref{D:1.1}. Then, we have
\begin{gather}
-\iint_{Q\X Q} |u-v|(\xi_t+\xi_s)\notag\\ 
+\iint_{Q\X Q} \nabla_{x''}\cdot \Bbf(u,v)(\nabla_{x''}\xi+\nabla_{y''}\xi)-\iint_{Q\X Q}F(u,v)(\nabla_x\xi+\nabla_y\xi)\label{e6.1}\\
+\iint_{Q\X Q} \nabla_{y''}\cdot \Bbf(u,v)(\nabla_{x''}\xi+\nabla_{y''}\xi)\le 0.\notag
\end{gather}
\end{lemma}

\begin{theorem}\label{T:6.1} There exists at most one entropy solution to \eqref{e1.1}--\eqref{e1.4}, according to Definition~\ref{D:1.1}.
\end{theorem}
  
 \begin{proof} We consider two entropy solutions $u$ and $v$, where $u=u(t,x)$, $v=v(s,y)$.  Following closely the lines of the proof of theorem~1.1 of \cite{MPT},  we will first prove that 
 \begin{equation}\label{e6.2}
 -\int_{Q}|u-v|\tilde\varphi_t-K_{x''}(u,v)\cdot \nabla \tilde \varphi \,dx\,dt\le 0,
 \end{equation}
 for all $0\le\tilde\varphi\in C_0^\infty((0,T)\X\Om'\X \R^{d''})$. Once we get inequality \eqref{e6.2}, the conclusion is quite simple, using Theorem~\ref{T:1.2}. We take $\tilde\varphi(t,x):=\theta(t)\z_\l'(x')$, in $Q$, where $\z_\l'(x')$ is a $\Om'$-boundary layer sequence, and $0\le \theta (t)\in C_0^\infty(0,T)$. Making then $\l\to0$ and using the strong trace property at $\po\Om'\X\Om''$, which holds for both $u$ and $v$, by Theorem~\ref{T:1.2}, we arrive at
 \begin{equation*}
 -\int_{Q}|u-v|\theta'(t)\,dx\,dt\le 0,
 \end{equation*}
 and then, as usual, taking $\theta$ as an approximation of the characteristic function of a time interval $(\tau,t)$, we obtain
 \begin{equation*}
 \int_{\Om}|u(t,x)-v(t,x)|\,dx\le \int_{\Om}|u(\tau,x)-v(\tau,x)|\,dx
 \end{equation*}
 which, by making $\tau\to0$, gives the desired result. 
 
 So, it remains to prove \eqref{e6.2}.  As we show next, the proof of \eqref{e6.2} follows closely the lines of the  proof of theorem~1.1 in \cite{MPT}, only  simple adaptations being 
 required. 
 
 Following \cite{MPT}, we use the notation $Q_{(t,x)}$ or $Q_{(s,y)}$ to emphasize the domain of integration, whether with respect to $(t,x)$ or with respect to $(s,y)$, respectively, and we use simply $Q\X Q$ as the domain of integration of the four variables.

 To begin with,  we take  smooth functions $\psi_1''(x'')$, $\psi_2''(x'')$, $0\le\psi_i''\le1$, $i=1,2$,   with support in a ball $B''$,  as in \eqref{e1.8B},  around a point of the boundary $\po\Om''$, and here we define $\psi(x,y)=\psi_1(x)\psi_2(y)$ with $\psi_1(x)=\phi_1'(x')\psi_1''(x'')$ and $\psi_2(y)=\phi_2'(y')\psi_2''(y'')$, for $\phi_1',\phi_2'\in C_0^\infty(\Om')$,  where  
 $\psi_2''(x'')=1$  for $x''\in\supp\psi_1''$,  and we use the notation $x'=(x_1,\cdots,x_{d'})$, $x''=(x_{d'+1},\cdots,x_d)$,  $d=d'+d''$, so that $x=(x',x'')$.  
 
 As in \cite{MPT} we consider our coordinates $x,y$ already relabeled so that $\po\Om_{x''}''\cap B_{x''}''=\{x_d=\gamma(\bar x'')\}$,
 $\po\Om_{y''}''\cap B_{y''}''=\{y_d=\gamma(\bar y'')\}$, and  we take $\xi(t,x,s,y)$ in \eqref{e6.1} in the form
 $$
 \xi(t,x,s,y)=\z_{d}''(x'')\z_\eta''(y'')\rho(t-s,x-y)\psi(x,y)\theta(t),
 $$
 with $\theta\in C_c^\infty(0,T)$, $\theta\ge0$, $\rho=\rho_{l,m,n}=\rho_l(t-s)\rho_m(\bar x-\bar y)\rho_n(x_d-y_d)$, where $\rho_l,\rho_n$ are sequences of symmetric mollifiers in $\R$ and $\rho_m$ is a sequence of symmetric mollifiers in $\R^{d-1}$ and we denote $\bar x=(x_1,\cdots,x_{d-1})$ so $x=(\bar x,x_d)$, $y=(\bar y,y_d)$.  
 
 As for $\z_\d''(x''),\z_\eta''(y)$, we assume that they are the corresponding $\Om''$-canonical local boundary layer sequences. The use of the canonical local boundary layer sequences avoids all rather technical analysis in the Appendix of \cite{MPT} to study the properties of their special smooth boundary layer sequence, whose aim is to prove that, after all, it behaves as the canonical local boundary layer sequence.  
 
 Denoting $\nabla_{x+y}=\nabla_x+\nabla_y$, from \eqref{e6.1} we have
 \begin{align*}
 &-\iint\limits_{Q\X Q}|u-v|\z_\d''\z_\eta''\rho\theta_t\psi+\iint\limits_{Q\X Q}\nabla_{x''}\cdot\Bbf(u,v)\z_\d''\z_\eta''\rho\theta\nabla_{x''+y''}\psi\\
 &+\iint\limits_{Q\X Q}\nabla_{y''}\cdot \Bbf(u,v)\z_\d''\z_\eta''\rho\theta\nabla_{x''+y''}\psi-\iint_{Q\X Q}F(u,v)\z_\d''\z_\eta''\rho\theta\nabla_{x+y}\psi\\
&+\iint\limits_{Q\X Q}H_{x''}(u,v,a_0(x))\nabla_{x''}\z_\d''\z_\eta''\rho\theta\psi+\iint\limits_{Q\X Q}H_{y''}(u,v,a_0(y))\z_\d''\nabla_{y''}\z_\eta''\rho\theta\psi\\
&\le \iint\limits_{Q\X Q}K_{x''}(u,a_0(x))\nabla_{x''}\z_\d''\z_\eta''\rho\theta\psi+\iint\limits_{Q\X Q}K_{y''}(v,a_0(y))\z_\d''\nabla_{y''}\z_\eta''\rho\theta\psi\\
& -\iint\limits_{Q\X Q}\{\nabla_{x''}\cdot \Bbf(a_0(x),v)\nabla_{x''}\z_\d''\z_\eta''+\nabla_{y''}\cdot \Bbf(a_0(y),u)\z_\d''\nabla_{y''}\z_\eta''\}\rho\theta\psi\\
& +\iint\limits_{Q\X Q}\{F(a_0(x),v)\nabla_{x''}\z_\d''\z_\eta''+F(a_0(y),u)\z_\d''\nabla_{y''}\z_\eta''\}\rho\theta\psi\\
& -\iint\limits_{Q\X Q}\{\nabla_{x''}\cdot \Bbf(u,v)\z_\d''\nabla_{y''}\z_\eta''+\nabla_{y''}\cdot \Bbf(u,v)\nabla_{x''}\z_\d''\z_\eta''\}\rho\theta\psi\\
\end{align*}
Making $\d,\eta\to0$ and using \eqref{e1.11} for both $u$ and $v$, we get
 \begin{align*}
  &-\iint\limits_{Q\X Q}|u-v|\rho\theta_t\psi+\iint\limits_{Q\X Q}\nabla_{x''}\cdot\Bbf(u,v)\rho\theta\nabla_{x''+y''}\psi\\
 &+\iint\limits_{Q\X Q}\nabla_{y''}\cdot \Bbf(u,v)\rho\theta\nabla_{x''+y''}\psi-\iint_{Q\X Q}F(u,v)\rho\theta\nabla_{x+y}\psi\\
&\qquad \le \liminf_{\eta\to0}\liminf_{\d\to0} \{I_1+I_2+I_3+I_4+I_5\},
\end{align*}
where
\begin{align*}
&I_1=\iint\limits_{Q\X Q}K_{x''}(u,a_0(x))\nabla_{x''}\z_\d''\z_\eta''\rho\theta\psi,\\
&I_2=\iint\limits_{Q\X Q}K_{y''}(v,a_0(y))\z_\d''\nabla_{y''}\z_\eta'' \rho\theta\psi,\\
 &I_3=\iint\limits_{Q\X Q}\{F(a_0(x),v)\nabla_{x''}\z_\d''\z_\eta'' +F(a_0(y),u)z_\d''\nabla_{y''}\z_\eta'' \} \rho\theta\psi,\\
 &I_4=-\iint\limits_{Q\X Q}\{\nabla_{x''}\cdot \Bbf(u,v)\z_\d''\nabla_{y''}\z_\eta''+\nabla_{y''}\cdot \Bbf(a_0(y),u)\z_\d''\nabla_{y''}\z_\eta''\}\rho\theta\psi.\\
  &I_5=-\iint\limits_{Q\X Q}\{\nabla_{y''}\cdot \Bbf(u,v)\nabla_{x''}\z_\d''\z_\eta''+\nabla_{x''}\cdot \Bbf(a_0(x),v) \nabla_{x''}\z_\d''\z_\eta''\}\rho\theta\psi.
 \end{align*}

The integrals $I_1$, $I_2$ and $I_3$ are the easiest to deal with. Indeed, from Lemma~\ref{L:1.1}, we have that $\K(u,a_0):=(-|u-a_0|\psi_1, K_{x''}(u,a_0)\psi_1)\in\DM^2(Q)$, so that,  by \eqref{e11}, 
$$
\lim_{\eta\to0}\lim_{\d\to0} I_1=-\la\, \K(u,a_0)\cdot\nu,\int_{Q_{(s,y)}}\theta \rho(t-s,x-y)\psi_2\,\ra,
$$
where the latter denotes the normal trace of $\K(u,a_0)$ applied to $\int_{Q_{(s,y)}}\theta \rho(t-s,x-y)\psi_2$.

Similarly, we have
$$
\lim_{\eta\to0}\lim_{\d\to0} I_2=-\la\, \K(v,a_0)\cdot\nu,\int_{Q_{(t,x)}}\theta \rho(t-s,x-y)\psi_1\,\ra.
$$
Now, for $I_3$, since in each term  the boundary layer sequence is in the integral of a smooth function,  we immediately get
\begin{equation}\label{e6.F1}
\begin{aligned}
\lim_{\eta\to0}\lim_{\d\to0} I_3=&-\iint\limits_{\Sigma_{(t,x)}\X Q_{(s,y)}}F(a_0(x),v(s,y))\cdot\nu(x)\rho\theta\psi \\
            &-\iint\limits_{Q_{(t,x)}\X\Sigma_{(s,y)}}F(a_0(y),u(t,x))\cdot\nu(y)\rho\theta\psi,
            \end{aligned}
\end{equation}
where $\Sigma_{(t,x)}=((0,T)\X(\Om'\X\po\Om''))_{(t,x)}$ and $\Sigma_{(s,y)}=((0,T)\X(\Om'\X\po\Om''))_{(s,y)}$. 

Now, the integrals $I_4$ and $I_5$ are the ones posing the whole difficulty.  To begin with, we expand $I_4$ as follows, making an integration by parts in the first term,
\begin{equation}\label{e6.I4}
\begin{aligned}
I_4=&\iint\limits_{Q\X Q} \Bbf(u,v) \nabla_{x''}\z''_\d\nabla_{y''}\z_\eta''\rho\theta\psi\\
 &+\iint\limits_{Q\X Q} \Bbf(u,v) \z''_\d\nabla_{y''}\z_\eta''\rho\theta\nabla_{x''}\psi\\
 &+\iint\limits_{Q\X Q} \Bbf(u,v) \z''_\d\nabla_{y''}\z_\eta''\nabla_{x''}\rho\theta \psi\\
 &-\iint\limits_{Q\X Q} \nabla_{y''}\cdot\Bbf(a_0(y),u) \z''_\d\nabla_{y''}\z_\eta''\rho\theta\psi.
 \end{aligned}
\end{equation}
 We observe that, from \eqref{e1.8D'}, the first two integrals converge as $\d,\eta\to0$ to
 \begin{align*}
& \iint\limits_{\Sigma_{(t,x)}\X\Sigma_{(s,y)}}\Bbf(a_0(x),a_0(y))\rho\theta\psi\nu_x\nu_y, \\
 &-\iint\limits_{Q_{(t,x)}\X \Sigma_{(s,y)}} \Bbf(u,a_0(y))\rho\theta\nabla_{x''}\psi\nu_y,
 \end{align*}
both of which clearly go to zero when we make $l\to\infty$, $m\to\infty$ and then $n\to\infty$, where we also use  \eqref{e1.8D'} in the second integral. 

The last two integrals  require a more delicate analysis. The presence of $\nabla_{x''}\rho$ in the first of the last two integrals is a red flag. The second of these two last integrals is nice since, as a function of $y$, $\nabla_{y''}\cdot \Bbf(a_0(y),v)$ is $BV$, so we can take its limit as $\d,\eta\to0$ to find
\begin{equation*}
\iint\limits_{Q_{(t,x)}\X \Sigma_{(s,y)}}\nabla_{y''}\cdot\Bbf(a_0(y),u)\nu  \rho\theta \psi,
\end{equation*}
and we recall that $a_0$ depends only on $\bar y$, where $\bar y=(y_1,\cdots,y_{d-1})$. If $U_{(s,y)}\subset\R^{d-1}$ is such that $\Sigma_{(s,y)}$ is the graph of $\tilde \gamma(t,\bar y)=\gamma(\bar y'')$, over $U_{(s,y)}$, we may rewrite the last integral as
 \begin{equation}\label{e6.I4'}
 -\iint\limits_{Q_{(t,x)}\X U_{(s,y)}}\tilde \nabla_{\bar y''}\cdot\Bbf(a_0(y),u) N \rho(\bar x-\bar y,x_d-\gamma(\bar y'')) \theta \psi,
\end{equation}
with $\tilde\nabla_{\bar y''}=\binom{\nabla_{\bar y''}}{0}$, $N=(-\nabla_{\bar y''}\gamma,1)$, where we use the fact that the unit normal to $\po\Om''$ is $\nu=\frac1{\sqrt{1+|\nabla\gamma|^2}} (-\nabla_{\bar y''}\gamma,1)$ and the Jacobian is $\sqrt{1+|\nabla\gamma|^2}$.
The striking observation in \cite{MPT} , recalling that $\nabla_{x''}\rho=-\nabla_{y''}\rho$, is that the last integral may be used to provide a cancelation of the integral involving $\nabla_{\bar y''}\rho$ coming out in the limit when $\d,\eta\to0$ of the third integral in \eqref{e6.I4},   leaving only the term involving $\frac{\partial \rho}{\partial y_d}$, which allows to make first $l,m\to\infty$, and then make $n\to\infty$. Namely, concerning the third integral in \eqref{e6.I4}, first we observe that 
$$
\z''_\eta(y'')=z_{\eta}(\gamma(\bar y'')-y_d),\quad \text{where}\quad z_\eta(r)=\begin{cases}0, &\text{for $r<0$},\\  r/\eta, &\text{for $0\le r<\eta$},\\ 1, &\text{for $r\ge \eta$}. \end{cases}
$$
 Therefore, $\nabla_{y''}\z_\eta(y'')=-z'_\eta(\gamma(\bar y'')-y_d)N$,  and $z_\eta'$ is the derivative of $z_\eta$ which clearly converges when $\eta\to0$ to $\d_{\{0\}}$, where the latter is the Dirac measure concentrated at the origin.  So, the limit of the third integral in \eqref{e6.I4} when $\d,\eta\to0$ is
 \begin{multline}\label{e6.I4''}
  -\iint\limits_{Q_{(t,x)}\X U_{(s,y)}}\Bbf(u, a_0(y))N \tilde\nabla_{\bar y''}\rho(\bar x-\bar y,x_d-\gamma(\bar y''))\theta\psi \\
  -\iint\limits_{Q_{(t,x)}\X U_{(s,y)}}\Bbf(u, a_0(y))\cdot N\otimes N \frac{\partial\rho}{\partial x_d}(\bar x-\bar y,x_d-\gamma(\bar y''))\theta\psi.
\end{multline}
Integrating by parts in the first integral in \eqref{e6.I4''}, results one integral which is the negative of the one in \eqref{e6.I4'}, and  two others that clearly vanish in the limit when $l,m,n\to\infty$, by \eqref{e1.8D'}. Therefore, the only relevant term coming from \eqref{e6.I4} is the second integral in \eqref{e6.I4''}. Now, after sending $l,m\to\infty$, this integral becomes
\begin{equation*}
  -\int\limits_{Q_{(t,x)}}\Bbf(a_0(x),u)\cdot N\otimes N \rho_n'(x_d-\gamma(\bar x''))\theta\psi,
  \end{equation*}
Adding it to the one resulting from taking the limit $l,m\to\infty$ in the second integral in \eqref{e6.F1}, we get
 \begin{equation}\label{e6.K1}
M_2= -\int\limits_{Q_{(t,x)}}\{\Bbf(a_0(x),u)\cdot N\otimes N \rho_n'+ F(a_0(x),u)N(\bar x'')\rho_n\}\theta\psi,
 \end{equation}
 where $\rho_n'$ and $\rho_n$ are evaluated at $x_d-\gamma(\bar x'')$. Now, as in  \cite{MPT}, define $\om_n:=2\int_{x_d-\gamma(\bar x'')}^0\rho_n(s)\,ds$, which is clearly a boundary layer sequence for $\{x''\in\R^{d''}\,:\, x_d<\gamma(\bar x'')\}\supset \Om''\cap B''$, and satisfies
 $$
 -\nabla_{x''}^2\om_n=2N\otimes N \rho_n'(x_d-\gamma(\bar x''))-2 \nabla_{x''}^2\gamma(\bar x'')\,\rho_n(x_d-\gamma(\bar x'')).
 $$
Substituting in \eqref{e6.K1}, observing that $\nabla\om_n=-2\rho_n N(\bar x'')$ we get
\begin{multline*}
\frac12\int\limits_{Q_{(t,x)}}\{\Bbf(u,a_0)\cdot \nabla_{x''}^2\om_n+F(a_0,u)\nabla\om_n\}\theta \psi\\
-\int_{Q_{(t,x)}} \Bbf(u,a_0)\cdot\nabla_{x''}^2 \gamma(\bar x'') \rho_n\theta\psi.
\end{multline*}
Now integrating by parts the first term in the first integral above, using \eqref{e1.8D'} again, we get
\begin{align*}
M_2=&-\frac12\int_{Q_{(t,x)}}K_{x''}(u,a_0)\nabla \om_n\theta \psi\\
&-\frac12\int_{Q_{(t,x)}}\Bbf(u,a_0)\nabla\om_n\theta \nabla\psi\\
&-\int_{Q_{(t,x)}}\Bbf(u,a_0)\cdot\nabla^2\gamma\rho_n\theta\psi.
\end{align*}
Hence, using \eqref{e1.8D'} for the second and the third integral,  passing to the limit  when $n\to\infty$ we get
$$
\lim_{n\to\infty}M_2=\frac12\la \K(u,a_0)\cdot\nu,\theta\psi_2\ra.
$$
Proceeding with $I_5$ exactly as it was done for $I_4$, using this time the limit as $l,m\to\infty$ of the second integral in \eqref{e6.F1}, we define $M_3$ as the analog of $M_2$
and get
$$
\lim_{n\to\infty}M_3=\frac12\la \K(v,a_0)\cdot\nu,\theta\psi_1\ra.
$$
Finally, defining $M_1$ as the limit as  $l,m\to\infty$ of $\lim_{\d,\eta\to 0} (I_1+I_2)$, we clearly obtain
$$
\lim_{n\to\infty} M_1= -\frac12\la \K(u,a_0)\cdot\nu,\theta\psi_2\ra-\frac12\la \K(v,a_0)\cdot\nu,\theta\psi_1\ra,
$$
that is, $M_1+M_2+M_3=0$, which then finishes the proof.

 \end{proof}

\section{Possible Conflicts of Interest Information}
H.~Frid gratefully acknowledges the support from CNPq, through grant proc. 303950/2009-9, and FAPERJ, through grant E-26/103.019/2011. Y.~Li gratefully acknowledges the support from NSF of China, through grants 11231006 and 11571232, NSF of Shanghai, through grant 14ZR1423100, and Shanghai Committee of Science and Technology, through grant 15XD1502300.

 \end{document}